%% file: pap.tex
\documentclass[11pt]{article}

\usepackage{fullpage}
\usepackage{hyperref}
\usepackage{url}
\usepackage{graphicx}
\usepackage{setspace}
\usepackage{epstopdf}
\usepackage{tikz}
\usetikzlibrary{intersections}
\usetikzlibrary{shapes,arrows}
\usetikzlibrary{calc}
\tikzstyle{line}=[draw]
\usepackage{todonotes}

\usepackage{enumerate}
\usepackage{natbib}
\usepackage{amssymb,amsmath,amsthm,amsfonts}
\usepackage{nicefrac,mathtools}
\usepackage{algorithm,algorithmic}
\usepackage{verbatim}

\usepackage{bm}

\newtheorem{theorem}{Theorem}
\newtheorem{corollary}[theorem]{Corollary}
\newtheorem{lemma}[theorem]{Lemma}

\newtheorem{definition}{Definition}
\newtheorem{remark}[theorem]{Remark}

\newcounter{assumption}
\renewcommand{\theassumption}{A\arabic{assumption}}

\newenvironment{ass}[1][]{\begin{trivlist}\item[] \refstepcounter{assumption}%
 {\bf Assumption\ \theassumption\ {\em (#1)} } }{
 \ifvmode\smallskip\fi\end{trivlist}}

\newcommand{\df}{\stackrel{\mbox{\rm\tiny def}}{=}}
\input{def}

\title{Large-Scale Markov Decision Problems via the Linear Programming Dual}

\author{
Yasin Abbasi-Yadkori\\
Adobe Research\\
\and
Peter L. Bartlett \\
UC Berkeley\\
\and
Xi Chen\\
NYU\\
\and
Alan Malek\\
Simons Institute\\
}

\begin{document}

\maketitle


\begin{abstract}
  We consider the problem of controlling a fully specified Markov decision process (MDP), also known as the planning problem, when the state space is very large and calculating the optimal policy is intractable. Instead, we pursue the more modest goal of optimizing over some small family of policies. Specifically, we show that the family of policies associated with a low-dimensional approximation of occupancy measures yields a tractable optimization. Moreover, we propose an efficient algorithm, scaling with the size of the subspace but not the state space, that is able to find a policy with low \emph{excess loss} relative to the best policy in this class. To the best of our knowledge, such results did not exist in the literature previously. We bound excess loss in the average cost and discounted cost cases, which are treated separately. Preliminary experiments show the effectiveness of the proposed algorithms in a queueing application.  
\end{abstract}

\section{Introduction}
The  Markov Decision Process planning problem is to find a good policy given complete knowledge of the transition dynamics and loss function. Much work has been done by the reinforcement learning community; the earliest approaches with convergence guarantees date back to value iteration~\citep{Bellman-1957}, policy iteration~\citep{Howard-1960}, and other dynamic programming ideas. Another thread has been the linear programming formulation~\citep{Manne-1960}. In general, the planning problem is well understood for state-spaces small enough to permit computation of the value function
\citep{Bertsekas-2007}. However, in large state space problems, both the dynamic programming and linear program approaches are computationally infeasible as complexity scales quadratically with the number of states.

A popular approach to large-scale problems is to search for the optimal value function within the linear span of a small number of features with the hope that the optimal value function will be well approximated and will lead to a near optimal policy. Two popular methods are Approximate Dynamic Programming (ADP) and Approximate Linear Programming (ALP). For a survey on theoretical results for ADP, see \citep{Bertsekas-Tsitsiklis-1996}, \citep[Vol. 2, Chapter 6]{Bertsekas-2007}, and more recent papers \citep{Sutton-Szepesvari-Maei-2009, Sutton-Maei-Precup-Bhatnagar-Silver-Szepesvari-Wiewiora-2009, Maei-Szepesvari-Bhatnagar-Precup-Silver-Sutton-2009, Maei-Szepesvari-Bhatnagar-Sutton-2010}.

Our goal is to find an almost-optimal policy in some low dimensional space such that the complexity scales with the low dimensional space but is sublinear in the size of the state space. In contrast, all prior work on ALP either scales badly or requires access to samples from a distribution that depends on the optimal policy. To accomplish this, we will use randomized algorithms to optimize policies that are parameterized by linear functions in the dual LP. We provide performance bounds in the average loss and discounted loss cases. In particular, we introduce new proof techniques and tools for average cost and discounted cost MDP problems and use these techniques to derive a reduction to stochastic convex optimization with accompanying error bounds.

\subsection{Markov Decision Process}
Markov decision processes have become a popular approach to modeling an agent interacting with an environment, and, most notably, are the model assumed by reinforcement learning. Using $[N]\:=\{1,\ldots, N\}$, an MDP is parameterized by:
\begin{enumerate}
\item a discrete state space $\{1,2,\ldots, \cX\}$,
\item a discrete action space $\{1,2,\ldots, \cA\}$,
\item transition dynamics $P:[\cX]\times[\cA]\rightarrow\triangle_{[\cX]}$ that describes the distribution of the next states $x'$ given a current state and action $(x,a)$, and
\item loss function $\ell:[\cX]\times [\cA] \rightarrow [0,1]$ that provides the cost of taking an action in a given state.
\end{enumerate}
The (fully observed) state encapsulates all the persistent information of the environment, and the influence of the agent is captured through the transition distribution, which is a function of the current state and the current action.

A policy $\pi:[\cX]\rightarrow \triangle_{[\cA]}$ gives a distribution over actions for every possible state, and the goal of the learner is to identify a policy with small loss. Throughout, we will use $x$ and $a$ to refer to specific states and actions, respectively. Given some random variable $X_0$ for the starting distribution and some fixed policy $\pi$, the distribution of the random variable of the initial action $A_0$ is fixed. Then, given the transition dynamics $P$ and $\pi$, we can calculate the random trajectory $X_1, A_1, X_2, A_2,\ldots$. The random variables $X_t$ and $A_t$ will always refer to the random state and actions induced by a fixed policy $\pi$, the transition dynamics, and initial distribution of $X_0$. Using this random variable notation, we will write $P(X_{t+1} = x' | X_t = x, A_t = a)$ to refer to the $x'$th entry of $P(x,a)$, i.e. the probability of transitioning to state $x'$ from state $x$ when action $a$ is taken.

How can we evaluate a policy? The two most common metrics are average cost and discounted cost. Average cost is roughly the expected loss of the policy once the Markov chain has reached stationarity and disregards the transient dynamics. Discounted cost minimizes the cost where future losses $t$ rounds into the future are discounted by $\gamma^t$, where $\gamma\in(0,1)$ is some discounting factor. Therefore, discounted cost emphasized the short-term cost and roughly only considers $1/(1-\gamma)$ rounds into the future. Precisely,
\begin{align}
\lambda_\pi(x) &\df \lim_{n\rightarrow\infty}\ex\left[\frac{1}{n}\sum_{t=0}^n \ell(X_t,\pi(X_t)) \ \middle| \ X_0 = x\right] &\text{(average cost), and}\\
J_\pi(x) &\df \ex\left[\sum_{t=0}^\infty \gamma^t \ell(X_t,\pi(X_t)) \ \middle| \ X_0 = x\right] &\text{(discounted cost)}
\end{align}
The initial state is very relevant for $J$ but irrelevant for $\lambda$ under the usual regularity (it is sufficient to assume that the induced Markov chain is recurrent \citep{Puterman:MDP}). We study the average cost in Section~\ref{sec:average_cost} and the discounted cost in Section~\ref{sec:discounted_cost}.

\subsection{Notation}
It will be convenient to be able to write the transition dynamics as a matrix multiplication. For vectors $v\in\Reals^{\cX \cA}$ over state-action pairs, we will write $v(x,a)$ for the element corresponding to state $x$ and action $a$. The specific mapping from $(x,a)$ to $\{1,\ldots, \cX\cA\}$ is irrelevant, so just pick one and fix your favorite. We can then define the matrix $P \in \Reals^{\cX\cA\times \cX}$ to have row $P(x,a)$ in the $(x,a)$ position; therefore, if $v$ is a probability distribution over $X_t, A_t$, then $v^\top P \in \triangle_{\cX}$ is the distribution over $X_{t+1}$. We can also define the vector of losses $\ell$ to have value $\ell(x,a)$ in position $x,a$.

Given a vector $v$ and a matrix $M\in\Reals^{\cX\times\cX}$, we will use $v(i)$ for the $i$th component of vector $v$ and $M_{i,:}$, $M_{:,j}$, and $M_{ij}$ for the the $i$th row, $j$th column, and element in the $i,j$ position of $M$, respectively. For matrices $M\in\Reals^{\cX\cA\times\cX}$, where the first index is over state-action pairs, we will define $M_{(x,a),:}$ to be the row corresponding to $(x,a)$ and $M_{:,x}$ to be the column, over state-action pairs, corresponding to the $x$th column. 

Any distribution over state-action pairs $\mu$ defines a policy $\pi_\mu$ with
\begin{equation}
\pi_\mu(a|x) = \frac{\mu(x,a)}{\sum_{a'\in[\cA]}  \mu(x,a')},
\end{equation}
with $\pi_\mu(a|x) = \cA^{-1}$ if $\mu(x,a) =0$ for all $a$. This is simply the conditional distribution of $A$ given $X$.
We will also define the marginalization matrix $B\in \{0,1\}^{\cX\cA\times \cX}$ to be the binary matrix such that the $x$th coordinate of $v^\top B$ is $\sum_a v(x,a)$. If $v$ is a probability distribution over $X_t, A_t$, then $v^\top B$ is the marginal of $X_t$. 

For some fixed policy $\pi$, we would also like to refer to the induced state transition matrix, $P^\pi$, defined by 
\[
  (P^\pi)_{x,x'} = \sum_{a}P(X_{t+1}=x'|X_t=x,A_t=a)\pi(A_t=a|X_t=x),
\]
so that if $X_t\sim v$, then $X_{t+1}\sim {P^\pi}^\top v$ if policy $\pi$ is used.

We will use the norms $\norm{v}_{1,c} = \sum_{i} c_i\abs{v_i}$ and $\norm{v}_{\infty, c} = \max_i c_i \abs{v_i}$ (for a positive vector $c$). The constant one and zero vector are  $\mathbf 1$ and $\mathbf 0$, and  $\wedge$ and $\vee$ refer to the element-wise minimum and maximum. We can then compactly define $[v]_{-}=v \wedge 0$ and $[v]_+ = v \vee 0$ as the negative and positive parts of a vector $v$, respectively. Finally, $v \le w$ for two vectors means element-wise inequality, i.e. $v_i \le w_i$ for all $i$.

\subsection{Linear Programming for Average Cost}
For the average cost, let $h\in\Reals^\cX$ be a vector and $\lambda\in\Reals$ a scalar. The \emph{Bellman operator for average cost} is
\begin{equation*}
\L h(x)\df
\min_{a \in [\cA]} \left[ \ell(x,a) + \sum_{x'\in\cX} P_{(x,a),x'} h(x') \right],
\end{equation*}
and $h$ and $\lambda$ correspond to an optimal policy if they satisfy the Bellman optimality equation,
\begin{equation*}
\lambda + h(x) = \L h(x)\quad\forall x.
\end{equation*}
We will call such an $h$ and $\lambda$ the differential value function and the average cost, respectively. When the Bellman optimality equation is satisfied, the greedy policy (taking the action that achieves the minimum in the operator with probability 1) achieves the optimal loss~\citep{Puterman:MDP}. 

The Bellman optimality equation was first recast as a linear program by \cite{Manne-1960}, who noted that, if
$\lambda$ and $h$ satisfy $\L h\geq h+\lambda\ones$, then we must have $\lambda\leq\lambda^*$, where $\lambda^*$ is the average cost of the optimal policy. Therefore, the optimal $\lambda$ and $h$ are the solution to
\begin{align*}
&\max_{\lambda, h} \lambda\,, \\
\notag
&\mbox{s.t.}\quad h+\lambda\ones\leq \L h.
\end{align*}
Now, notice that $h(x)+\lambda \leq \min_{a}\left[\ell(x,a)+\sum_{y}P(y|x,a)h(y)\right]$ is equivalent to requiring
$h(x)+\lambda\leq\ell(x,a)+\sum_{y}P(y|x,a)h(y)$ for all $x$ and $a$. In our matrix notation, this is precisely $B(\lambda\ones+h)\leq\ell+P h$. Hence, the Bellman optimality equation is equivalent to the linear program
\begin{align}
\label{LP:exact_average}
&\max_{\lambda, h} \lambda\,, \\
\notag
&\mbox{s.t.}\quad B(\lambda \mathbf 1 + h) \le \ell + P h \,.
\end{align}
A standard computation shows that the dual of LP~\eqref{LP:exact_average} has the form of
\begin{align}
\label{LP:exact_average_dual}
&\min_{\mu\in \Reals^{\cX \cA}} \mu^\top \ell\,, \\
\notag
&\mbox{s.t.}\quad \mu^\top \ones = 1,\, \mu \geq \mathbf 0,\, \mu^\top (P - B) = \mathbf 0 \; ,
\end{align}
The dual variable, $\mu$, has an important interpretation: it is a stationary distribution over state-action pairs under its implied policy.

The first two constraints ensure that $\mu$ is a probability distribution over state-action space and the third constraint forces $\mu$ to be a stationary distribution under $\pi_{\mu}$. Intuitively, if $X_t\sim \mu^\top B$, then $\mu^\top P$ is the distribution of $X_{t+1}$ under policy $\pi_{\mu}$; hence, the third constraint implies that $X_t$ and $X_{t+1}$ have the same distribution. If $\mu$ is a stationary distribution, then the average loss under $\mu$ is exactly $\mu^\top \ell$.

\subsection{Linear Programming for Discounted Cost}
There are analogous notions for the discounted cost setting. We define a value function $J:  [\cX]\rightarrow \Reals$ as a mapping from states to discounted costs. The hope is to find $J^*$, where $J^*(x)$ is the discounted cost starting in state $x$ if the optimal policy is used.

We define the \emph{Bellman operator for discounted cost}
\begin{equation*}
\L^\gamma J(x)\df
\min_{a \in [\cA]} \left[ \ell(x,a) + \gamma \sum_{x'\in[\cX]} P_{(x,a),x'} J(x') \right]
\end{equation*}
and the optimal value function will be the fixed point of the Bellman operator,
\begin{equation*}
  L^\gamma J^* = J^*.
\end{equation*}
It is easy to check that $J\leq \L^\gamma J$ implies $J\leq J^*$, and therefore, for any strictly positive vector $\alpha \in\Reals^\cX$, the optimal value function is the solution to the linear program
\begin{align}
\label{LP:exact_discounted}
&\max_{J} \alpha^\top J\\
\notag
&\mbox{s.t.}\quad \L^\gamma J\geq J.
\end{align}

We also have an interpretable dual LP. Let $\alpha$ be such that $\alpha\ge 0$ and $\alpha^\top \mathbf{1} = 1$. The linear program for discounted MDPs in the dual space has the form of
\begin{align}
\label{LP:exact_discounted_dual}
&\min_{\nu\in\Reals^{\cX \cA}} \nu^\top \ell\,, \\
\notag
&\mbox{s.t.} \quad (B - \gamma P)^\top \nu  = \alpha,\quad \nu\geq 0, \quad \nu^\top \ones = \frac{1}{1-\gamma}.
\end{align}

Unlike the average cost case, the dual variable $\nu$ cannot be interpreted as a stationary distribution. However, it can be thought of as the discounted number of visits, as made explicit in the  following theorem from \citet{Puterman:MDP}:
\begin{theorem}
\label{thm:dual-lp-mdp}
\begin{enumerate}
\item  For each randomized Markovian policy $\pi$ and state $x$ and action $a$, define $\nu_\pi(x,a)$ by
\[
\nu_\pi(x,a) = \sum_{x'} \alpha (x') \sum_{t=1}^\infty \gamma^{t-1} P^\pi (x_t = x, a_t = a \ | \ x_1 = x') \;.
\]
Then $\nu_\pi$ is a feasible solution to the dual problem.
\item Suppose $\nu$ is a feasible solution to the dual problem, then, for each $x\in[\cX]$, $\sum_a \nu(x,a) > 0$. Define the randomized stationary policy $\pi_\nu$ by
\[
\pi_\nu(a | x) = \frac{\nu(x,a)}{\sum_{a'}\nu(x,a')} \;.
\]
Then, $\nu_{\pi_\nu}$ is a feasible solution to the dual LP and $\nu_{\pi_\nu} = \nu$.
\end{enumerate}
\end{theorem}
Thus, we can approximately solve the planning problem if we find a vector $z$ such that the discounted cost of the policy defined by $z$, namely $\ell^\top \nu_{\pi_z}$, is small. To handle possibly negative entries of $z$, we more generally define
\[
\pi_z(a | x) = \frac{\left[z(x,a)\right]_+}{\sum_{a'}\left[z(x,a')\right]_+} \;.
\]
In this case, the precise relationship between $\nu_{\pi_z}$ and the value function can be found in  \cite{Puterman:MDP}: for any vector $z$,
\begin{equation}
  \sum_{x,a} \nu_{\pi_{z}}(x,a)=\frac{1}{1-\gamma} \quad\text{and}\quad
  \nu_{\pi_z}^T \ell = \alpha^T J_{\pi_z},
  \label{eq:primal_dual_rel}
\end{equation}
where $J$ is the value function corresponding to policy $\pi_z$.

\subsection{Approximate Linear Programming}
If we ignore computational constraints, we can solve the planning problem by solving the linear programs  \eqref{LP:exact_average_dual} and \eqref{LP:exact_discounted_dual}. Unfortunately, state spaces are frequently very large and often grow exponentially with the complexity of the system (e.g. number of queues in the queuing network), and therefore any method polynomial in $\cX$ becomes intractable. The general method of solving the planning problem with an approximate solution to the linear program is called Approximate Linear Programming (ALP). As any general optimality guarantee is impossible with computation sublinear in $\cX$ without special knowledge of the problem, we instead aim for optimality with respect to some smaller policy class.

We take the less common approach of reducing the dimensionality by placing a subspace restriction of the dual variables. Let $\Phi\in\Reals^{\cX\cA\times d}$ by a feature matrix and $\mu_0$ some known stationary distribution (that can be taken to be zero but allows a user to start with a good policy). For the average cost case, we will limit our search to $\mu = \mu_0 + \Phi\theta$ for $\theta\in\Theta\subset\Reals^d$; that is, we will study the \emph{approximate average cost dual LP},
\begin{align}
\label{eq:dual-apprx}
&\min_{\theta\in\Theta}~ (\mu_0 + \Phi \theta)^\top  \ell\,, \\
\notag
&\mbox{s.t.}\quad (\mu_0+\Phi\theta)^\top  \mathbf 1 = 1,\, \mu_0+ \Phi \theta \geq \mathbf 0,\, (\mu_0+\Phi\theta)^\top (P - B) = \mathbf 0 \; .
\end{align}
we will only consider $\theta$ that sum to 1 and will restrict $\Theta$ to lie in $\{x\in\Reals^d: x^\top\ones = 1\}$. This restriction is without loss of generality, since we may always renormalize $\Phi$. 

For every $\theta$, we associate a policy
\begin{equation}
\label{eq:average_cost_policy}
\pi_{\theta}(a|x) = \frac{[\mu_0(x,a) + \Phi_{(x,a),:} \theta]_{+}}{\sum_{a'} [\mu_0(x,a') + \Phi_{(x,a'),:} \theta]_{+}}
\end{equation}
and a stationary distribution $\mu_\theta$ the actual stationary distribution of running policy $\pi_\theta$. Thus, the average cost corresponding to the policy $\pi_\theta$ is $\ell^\top\mu_\theta$.

For the discounted cost case with feature matrix $\Phi$, we restrict the dual variable to $\nu = \Phi\theta$ and define the \emph{approximate discounted cost dual LP}
\begin{align*}
&\min_{\theta\in\Theta}~ \ell^\top\Phi\theta\,, \\
\notag
&\mbox{s.t.} \quad(\Phi\theta)^\top \ones = \frac{1}{1-\gamma},\quad (B - \gamma P)^\top \Phi\theta  = \alpha,\quad \Phi\theta\geq 0.
\end{align*}
For every $\theta$, we define a policy
\begin{equation}
\label{eq:discounted_cost_policy}
\pi_{\theta}(a|x) = \frac{[\Phi_{(x,a),:} \theta]_{+}}{\sum_{a'} [\Phi_{(x,a'),:} \theta]_{+}},
\end{equation}
and let $\nu_\theta$ be the corresponding dual variable (i.e. the discounted number of visits); hence,  $\ell^\top\nu_\theta$ is the discounted cost as in \eqref{eq:primal_dual_rel}. In the discounted case, we will restrict $\Theta$ to lie in $\{x\in\Reals^d: x^\top\ones = (1-\gamma)^{-1}\}$.

\subsection{Problem Definition}
The goal of the paper is to find a $\theta$ such that the associated policy $\pi_{\theta}$ is close to the policy corresponding with the best $\theta \in\Theta$ in an efficient manner and while avoiding complexity proportional to $\cX$. This goal is formalized by the following definition.
\begin{definition}[Efficient Large-Scale Dual ALP]
\label{defn:ELALP}
For an MDP specified by $\ell$ and $P$ with the dual variables $\xi_\theta$ corresponding to $\theta\in\Theta$,
the efficient large-scale dual ALP problem is to find a $\widehat\theta$ such that
\begin{equation}\label{eqn:ELALP}
\ell^\top\xi_{\widehat \theta}
  \leq \min\left\{ \ell^\top\xi_\theta : \text{$\xi_\theta$ feasible
  for \eqref{LP:exact_average_dual} or \eqref{LP:exact_discounted_dual}
}\right\} + O(\epsilon)
\end{equation}
in time polynomial in $d$ and $1/\epsilon$. The model of computation
allows access to arbitrary entries of $\Phi$, $\ell$,
$P$, $\mu_0$, $P^\top\Phi$, and $\ell^\top\Phi$ in unit time.
\end{definition}
The computational complexity cannot scale with $\cX$ and we do not assume
any knowledge of the optimal policy. In fact, as we shall see, we
solve a harder problem, which we define as follows.

\begin{definition}[Expanded Efficient Large-Scale Dual ALP]
\label{defn:E.ELALP}
Let $V:\Reals^d\to\Reals_+$ be some ``violation function'' that
represents how far $\xi_\theta$ is from satisfying the constraints of \eqref{LP:exact_average_dual} or \eqref{LP:exact_discounted_dual} and has $V(\theta)=0$ if $\theta$ is feasible.

The expanded efficient large-scale dual ALP
problem is to produce parameters $\widehat\theta$ such that
\begin{equation}
\label{eqn:E-ELALP}
\ell^\top\xi_{\widehat \theta} \leq  \min_{\theta\in\Theta} \{ \ell^\top\xi_\theta
+V(\theta)\} +O(\epsilon),
\end{equation}
in time polynomial in $d$ and $1/\epsilon$, under the same model of
computation as in Definition~\ref{defn:ELALP}.
\end{definition}

Note that the expanded problem is strictly more general as guarantee
\eqref{eqn:E-ELALP} implies guarantee \eqref{eqn:ELALP}. Also, many
feature vectors $\Phi$ may not admit any feasible points. In this
case, the dual ALP problem is trivial, but the expanded problem is
still meaningful.

In particular, we desire an agnostic learning guarantee, where the true average cost of running the policy corresponding to $\widehat\theta$ to be close to the true average cost of the best policy in the class, regardless of how well the policy class models the optimal value function. To the best of our knowledge, such a guarantee does not exist in the literature.

Having access to arbitrary entries of the quantities in
Definition~\ref{defn:ELALP} arises naturally in many situations.
In many cases, entries of $P^\top\Phi$ are easy to compute.
For example, suppose that for any state $x'$ there are a small number of state-action pairs $(x,a)$ such that $P(x'|x,a)>0$. Consider Tetris; although the number of board configurations is large, each state has a small number of possible neighbors. Dynamics specified by graphical models with small connectivity also satisfy this constraint. Computing entries of $P^\top\Phi$ is also feasible given reasonable features. If a feature $\phi_i$ is a stationary distribution, then $P^\top\phi_i=B^\top \phi_i$. Otherwise, it is our prerogative to design sparse feature vectors, hence making the multiplication easy. We shall see an example of this setting later.

\subsection{Related Work}
Approximate linear programming, proposed by \cite{Schweitzer-Seidmann-1985}, constrained the value function in the linear program to a low-dimensional subspace. In the discounted cost setting, the first theoretical analysis of ALP methods, by \cite{DeFarias-VanRoy-2003}, analyzed the discounted primal LP \eqref{LP:exact_discounted_dual} performance
when only value functions of the form $J=\Psi w$, for some feature matrix $\Psi$, are considered. Roughly, they show that the ALP solution $w^*$ has the family of error bound indexed by a vector $u\in\Reals^\cX$
\[
\norm{\J^*-\Psi w^*}_{1,c}\leq \inf_{w}\frac{2c^\top u}{1-\gamma\beta(u)}\norm{\J^*-\Psi w},
\]
where $c$ is a ``state-relevance'' vector and $\beta_u = \gamma \max_{x,a} \sum_{x'} P_{(x,a), x'} u(x')/u(x)$ is a ``goodness-of-fit'' parameter that measures how well $u$ represents a stationary distribution. Unfortunately, $c$ and $u$ are typically hard to choose (for example, a good choice of $c$ would be the stationary distribution under $w^*$, which we do not know); but more importantly, the bound can be vacuous if $\Psi$ does not model the optimal value function well and $\norm{J_* - \Psi w}$ is always large. In particular, the problem we are considering in Definition~\ref{defn:E.ELALP} requires an additive bound with respect to the optimal parameter.

There are also computational concerns with the ALP, as the number of constraints remains $O(\cX\cA)$. One solution, proposed by \cite{DeFarias-VanRoy-2004}, was to sample a small number of constraints and solve the resulting LP; this resulted in an error bound of the form
\[
  \norm{J_* - \Psi \widehat w}_{1,c} \le \norm{J_* - \Psi  w_*}_{1,c} + \epsilon \norm{J_*}_{1,c},
\]
but the required number of sampled constraints needs to be a function of the stationary distribution of the optimal policy.

\cite{Desai-Farias-Moallemi-2012} proposed a different relaxation by defining the \emph{Smoothed Approximate Linear Program}, which only requires a soft feasibility and solves the linear program
\begin{align*}
&\max_{J} c^\top \Psi w\\
\notag
&\mbox{s.t.}\quad \L^\gamma \Psi w + s \geq \Psi w,\; s \geq 0, \;\nu^\top_{\pi_{*},\alpha} s\leq D
\end{align*}
which is exact LP \eqref{LP:exact_discounted} with $J = \Psi w$, the Bellman optimality constraint relaxed with a slack variable $s$, and additional bounds places on $s$. Here, $\nu^\top_{\pi_{*},\alpha}$ is the stationary distribution of the optimal policy and $D$ a violation budget, so the method requires some knowledge of the optimal policy. Despite this, the method remains computationally efficient and able to produce an agnostic approximation bound
\[
  \Vert J_* - \Psi w_*\Vert \leq \inf_w \Vert J_* - \Psi w\Vert O(1).
\]
However, their results do not easily extend to bounding the true error of running the policy associated with $w^*$, $\norm{J_{\nu_{\Psi w_*}}-J_*}$, without choosing $c$ as a function of $w^*$, which is itself a function of $c$. \cite{Petrik-Zilberstein-2009} proposed two different constraint relaxations schemes for the ALP, but did not show better approximations to the true solution, but rather focused on the better empirical performance. Yet another relaxation of the primal LP was proposed by \cite{lakshminarayanan2018linearly}, who generalized previous constraint sampling approaches. A bound for the discounted loss of the policy associated with the solution to this relaxed LP is presented and neatly decomposes into an estimation error that tends to zero and an approximation error between the optimal LP solution and the optimal relaxed LP solution.

In the average cost setting, largely thought to be more difficult, shares a similar history is that the first theoretical analysis for ALP was by \cite{deFarias-VanRoy-NIPS-2003}. They proposed a two stage LP. The first  approximates the optimal average cost and the second uses this estimate to try and learn the differential cost function $h$. The method suffers from the same problem as the discounted cost case is that we can only guarantee that $\lambda_{\hat w} - \lambda^*$, the excess loss of running the policy associated with $\hat w$, is small when we tune the LP with knowledge of $\mu_{\hat w}$, the stationary distribution.

Subsequent work in \cite{DeFarias-VanRoy-2006} took a different approach by viewing the average cost LP as a perturbed discounted cost LP, which is easier to analyze. Again, the span of the feature vectors needs to approximate the optimal policy in order for the excess loss guarantee to be meaningful. More recently, \cite{Veatch-2013} proposed a relaxation, similar to the smoothed ALP but with the total constraint violation terms entering the objective instead of facing a hard constraint, and derived similar loss bounds. 

Recently, \cite{chen2018scalable} analyzed a linearly parameterized ALP where the state and action spaces are both parameterized by linear features and the value function is assumed to be well approximated by linear function of the state features. They propose an efficient algorithm but suffer the same drawback and retain an error term of the form $\min_w\norm{\Psi w - J^*}$. Additionally, \cite{banijamali2018optimizing} study the related problem of optimizing policies in the convex hull of base policies. This problem can be seen as a special case of the usual ALP formulation when all features correspond to the stationary distribution of policies. 

To the best of our knowledge, no work has been able to show a bound of the form \eqref{eqn:E-ELALP}, as all the previous bounds are only meaningful when the approximate policy class can closely approximate the optimal policy. We are also the first to prove theoretical guarantees when the dual variables of the LP are restricted to a linear class, though such a parameterization appeared previously by \cite{Wang-Lizotte-Bowling-Schuurmans-2008}, albeit without theoretical guarantees. See Section~\ref{sec:related_work} in the appendix for a more thorough literature review and precise statements of prior bounds.

\subsection{Our Contributions}
We prove that if we parameterize the policy space by using the approximate dual LPs, then we can solve the expanded efficient large-scale dual ALP problem for both average cost and discounted cost. In the average cost setting, we require a (standard) assumption that the distribution of states under any policy converges quickly to its stationary distribution, but no such assumption is needed in the discounted cost setting. We also show that it suffices to solve the approximate dual LPs by approximately minimizing a surrogate loss function equal to the sum of the objective and a scaled violation function. 

We begin with the average cost in Section~\ref{sec:average_cost} and prove that, for some parameter $H>0$, any $\epsilon >0$ and $\delta>0$, the excess loss bound
\begin{align*}
\mu_{\widehat\theta}^\top \ell \le& \min_\theta \mu_\theta^\top \ell + H V(\theta) + O\left(\frac{1}{H} \log\left(\frac1\delta\right)\right) + O(\epsilon)
\end{align*}
holds with probability at least $1-\delta$, where
$V(\theta) = \norm{[\mu_0 + \Phi\theta]_{-}}_1 + \norm{(P - B)^\top
  (\mu_0 + \Phi\theta)}_1$. The $V(\theta)$ term is zero for feasible
points (that is, points in the intersection of the feasible set of
LP~\eqref{eq:dual-apprx} and the span of the features). For points
outside the feasible set, these terms measure the extent of constraint
violations for the vector $\mu_0 + \Phi\theta$, which indicate how
well stationary distributions can be represented by the chosen
features.

However, optimizing the excess loss bound to obtain the guarantee of Definition~\ref{defn:E.ELALP} requires us to tune $H$ correctly (in particular, setting $H \approx V(\theta)^{-1/2}$). Unfortunately, the convex surrogate is not jointly convex in $\theta$ and $H$. In Section~\ref{sec:finding_H}, we present and analyze a \emph{meta-algorithm} that solves the convex surrogate for a grid on $H$ values and returns a $\widehat\theta$ that has
\begin{equation*}
  \ell^\top\mu_{\widehat\theta}
  \leq
    \min_{\theta\in\Theta} \ell^\top\mu_{\theta}+ O\left(\sqrt{V(\theta)}\right)
    +O(\epsilon).
  \end{equation*}
  We emphasize that this bound is on the loss of actually running the $\pi_\theta$ policy, which could differ from the surrogate used in the optimization, $\ell^\top(\mu_0 + \Phi\theta)$. The run-time, up to logarithmic factors, is $O(\epsilon^{-4})$ for both algorithms; we essentially can tune $H$ for a small logarithmic cost.

  As we have seen in the related works section, all previous guarantees for efficient ADP algorithms only had meaningful guarantees when the policy class closely approximates the true value function,
  and many algorithms required tuning (say, of the state relevance weights) with knowledge of the optimal policy or stationary distribution. These restrictions render previous guarantees meaningless in many modern reinforcement learning systems, where the optimal value function is completely unknown and it is hopeless to try to engineer features that can approximate it \citep{goodfellow2016deep}. Our algorithm have guarantees that are meaningful in this setting, as we can obtain near-optimal excess loss within the policy class; in fact, one can use the stationary distribution of existing policies (based on DQN, heuristics, etc.) as feature vectors and improve upon them.

We then turn to the discounted cost problem in Section~\ref{sec:discounted_cost}. We propose an algorithm and show that it guarantees a bound on the discounted cost of the form
\begin{align*}
  \ell^{\top} \nu_{\widehat\theta_T}
  &\leq
    \ell^{\top} \nu_{\theta} +
    \left( \frac{6}{1-\gamma} + H \right)
    V(\theta)
    +O\left(\frac{1}{H(1-\gamma)}\right)+O(\epsilon).
\end{align*}
Furthermore, the \emph{meta-algorithm}, with minimal modification, solves the Expanded Efficient Large-Scale Dual ALP problem by obtaining the bound 
\begin{equation*}
  \ell^\top\nu_{\theta_{\hat k}} \leq \min_\theta \ell^\top\nu_\theta
  + O\left(\sqrt{V(\theta)}\right)
  + O(\epsilon),
\end{equation*}
where the violation function for the discounted cost is $V(\theta) = \|\left[\Phi \theta \right]_{-}\|_1+\|(B-\gamma P)^T \Phi \theta  -\alpha\|$.
  
  Section~\ref{sec:experiments} then demonstrates the effectiveness of our method on a well studied example from queuing theory, the Rybko-Stolyar queue. We show that using two simple heuristic policies with a small number of simple features provides good performance.

\section{The Dual ALP for Average Cost}
\label{sec:average_cost}
  Is this section, we propose and analyze our solution to the Expanded large-scale MDP problem for average cost. As discussed in the introduction, there are two main challenges for solving the planning problem in its LP formulation: the optimization is in dimension $\cX$, and there are $O(\cX\cA)$ constraints, which is intractable in the large state-space setting.

  We solve the two challenges by projecting the dual LP onto a subspace and by approximately solving the optimization using stochastic gradient descent, respectively. Unlike previous approaches for the primal LP, we show that an approximate solution in the dual allows us to bound the excess loss, i.e.\ one that controls the error between our approximate solution and the best solution in some approximate policy class, and thereby solve Equation~(\ref{eqn:E-ELALP}). We also provide some interpretation of the approximations we make.

Recall that, for a matrix $\Phi$ and a known stationary distribution $\mu_0$ (which may be set to zero if no distribution is known), we defined the dual ALP
\begin{align*}
&\min_{\theta} (\mu_0+  \Phi \theta)^\top \ell\,, \\
\notag
&\mbox{s.t.}\quad (\mu_0+  \Phi \theta)^\top \mathbf 1 = 1,\, \mu_0 + \Phi \theta \geq \mathbf 0,\, (\mu_0+  \Phi \theta)^\top (P - B) = \mathbf 0 \;
\end{align*}
and associated every $\theta$ with the policy
\begin{equation*}
\pi_{\theta}(a|x) = \frac{[\mu_0(x,a) + \Phi_{(x,a),:} \theta]_{+}}{\sum_{a'} [\mu_0(x,a') + \Phi_{(x,a'),:} \theta]_{+}} \,.
\end{equation*}
We denote the stationary distribution of this policy $\mu_{\theta}$, which is only equal to $\mu_0+\Phi\theta$ if $\theta$ is in the feasible set.

\subsection{A Reduction to Stochastic Convex Optimization}
Unfortunately, the ALP \eqref{eq:dual-apprx} still has $O(\cX\cA)$ constraints and cannot be solved exactly. Instead, we will use the penalty method to form an unconstrained convex optimization that will act as a surrogate for the original problem and show that it is a finite sum, e.g. equal to $\sum_{i=1}^N f_i(\theta)$. Therefore, we can apply the extensive literature of solving finite sum problems with stochastic subgradient descent methods.

  To this end, for a constant $H\geq 1$, define the following convex cost function by adding a multiple of the total constraint violations to the objective of the LP~\eqref{eq:dual-apprx}:
\begin{equation}\label{eq:objective_function}
\begin{split}	
c(\theta) &\df \ell^\top (\mu_0 + \Phi \theta) + H \norm{[\mu_0 + \Phi\theta]_{-}}_1 + H \norm{(P - B)^\top (\mu_0 + \Phi\theta)}_1 \\
&= \ell^\top (\mu_0 + \Phi \theta) + H \norm{[\mu_0 + \Phi\theta]_{-}}_1+ H \norm{(P - B)^\top \Phi\theta}_1 \\
&= \ell^\top (\mu_0 + \Phi \theta) + H \sum_{(x,a)} \abs{ [\mu_0(x,a) + \Phi_{(x,a),:}\theta]_{-}}
+ H \sum_{x'} \abs{(P - B)_{:,x'}^\top \Phi \theta} \; .
\end{split}
\end{equation}
We justify using this surrogate function as follows. Suppose we find a near optimal vector $\widehat \theta$ such that $c(\widehat \theta) \le \min_{\theta\in\Theta} c(\theta) + O(\epsilon)$. We will prove
\begin{enumerate}
\item that $\norm{[\mu_0 + \Phi\widehat\theta]_{-}}_1$ and $\norm{(P -
B)^\top (\mu_0 + \Phi\widehat\theta)}_1$ are small and $\mu_0 + \Phi\widehat\theta$ is close to $\mu_{\widehat\theta}$
(Lemma~\ref{lem:V.to.stationary.distribution}), and
\item that $\ell^\top (\mu_0 + \Phi \widehat\theta) \le \min_{\theta\in\Theta} c(\theta) + O(\epsilon)$.
\end{enumerate}
As we will show, these two facts imply that with high probability, for any $\theta\in\Theta$,
\begin{align*}
\mu_{\widehat\theta}^\top \ell \le \mu_\theta^\top \ell &+ \frac{1}{\epsilon}  \norm{[\mu_0 + \Phi\theta]_{-}}_1+ \frac{1}{\epsilon} \norm{(P - B)^\top (\mu_0 + \Phi\theta)}_1 + O(\epsilon).
\end{align*}

Unfortunately, calculating the gradients of $c(\theta)$ is $O(\cX\cA)$. Instead, we construct unbiased estimators and use stochastic subgradient descent.
Let $T$ be the number of iterations of our algorithm, $q_1$ and $q_2$ be distributions over the state-action and state space, respectively (we will later discuss how to choose them), and $((x_t,a_t))_{t=1\dots T}$ and $(x_t')_{t=1\dots T}$ be i.i.d. samples from these distributions. At round $t$, the algorithm estimates subgradient $\nabla c(\theta)$ by
\begin{align}
\label{eq:grad-est}
g_t(\theta) &= \ell^\top \Phi - H \frac{\Phi_{(x_t,a_t),:}}{q_1(x_t,a_t)} \one{\mu_0(x_t,a_t)+\Phi_{(x_t,a_t),:} \theta < 0} + H  \frac{(P-B)_{:,x_t'}^\top \Phi}{q_2(x_t')} s((P-B)_{:,x_t'}^\top \Phi \theta).
\end{align}
This estimate is fed to the projected subgradient method, which in turn generates a vector $\theta_t$. After $T$ rounds, we average vectors $(\theta_t)_{t=1\dots T}$ and obtain the final solution $\widehat \theta_T = \sum_{t=1}^T \theta_t/T$. Vector $\mu_0 + \Phi \widehat \theta_T$ defines a policy, which in turn defines a stationary distribution $\mu_{\widehat \theta_T}$. The algorithm is shown in Figure~\ref{alg:SGD}.

\begin{figure}
\begin{center}
\framebox{\parbox{12cm}{
\begin{algorithmic}
\STATE \textbf{Input: } Constants $S$ and $H$, number of rounds $T$, step size $\eta$.  
\STATE Let $\Pi_{\Theta}$ be the Euclidean projection onto $\Theta$.
\STATE Initialize $\theta_1 = 0$.
\FOR{$t:=1,2,\dots, T$}
\STATE Sample $(x_t,a_t)\sim q_1$ and $x_t'\sim q_2$.
\STATE Compute subgradient estimate $g_t$ \eqref{eq:grad-est}.
\STATE Update $\theta_{t+1} = \Pi_{\Theta} (\theta_t - \eta_t g_t)$.
\ENDFOR
\STATE $\widehat \theta_T = \frac{1}{T}\sum_{t=1}^T \theta_t$.
\STATE Return policy $\pi_{\widehat \theta_T}$.
\end{algorithmic}
}}
\end{center}
\caption{The Stochastic Subgradient Method for Markov Decision Processes}
\label{alg:SGD}
\end{figure}

\subsection{Excess Loss bound}
We now turn towards proving the main result of this section, Theorem~\ref{thm:main_average}, which requires a (standard) assumption that any policy quickly converges to its stationary distribution.
\begin{ass}[Fast Mixing]
\label{ass:uniform-mixing}
For any policy $\pi$, there exists a constant $\tau(\pi)>0$ such that for all distributions $\mu$ and $\mu'$ over the state space, $\norm{\mu^\top P^\pi - \mu'^\top P^\pi}_1 \le e^{-1/\tau(\pi)} \norm{\mu - \mu'}_1$.
\end{ass}
Define
\begin{align*}
C_1 = \max_{(x,a)\in [\cX]\times [\cA]}\frac{\norm{\Phi_{(x,a),:}}}{q_1(x, a)}\,, \qquad C_2 = \max_{x\in [\cX]}\frac{\norm{(P - B)_{:,x}^\top \Phi}}{q_2(x)} \; .
\end{align*}
These constants appear in our excess loss bounds, so we would like to choose distributions $q_1$ and $q_2$ such that $C_1$ and $C_2$ are small. Several common scenarios permit convenient $C_1$ and $C_2$:
\begin{itemize}
\item \textbf{Sparseness of $P$}
  If there is $C'>0$ such that for any $(x,a)$ and $i$, $\Phi_{(x,a), i}\le C'/(\cX \cA)$ and each column of $P$ has only $N$ non-zero elements, then we can simply choose $q_1$ and $q_2$ to be uniform distributions and  
\begin{align*}
\frac{\norm{\Phi_{(x,a),:}}}{q_1(x, a)} \le C'\,, \qquad \frac{\norm{(P - B)_{:,x}^\top \Phi}}{q_2(x)} \le C' (N + \cA) \; .
\end{align*}
\item \textbf{Features as stationary distributions}
  If every feature is the stationary distribution of some policy, then we can choose 
  $q_1(x,a)\propto \min\{\Phi_{(x,a),y}: y \in \mathcal X\}$ and $\norm{(P - B)_{:,x}^\top \Phi}$ vanishes.
\item \textbf{Exponential distributions}
  If $\Phi_{:, i}$ are exponential distributions and feature values at neighboring states are close to each other, then we can choose $q_1$ and $q_2$ to be appropriate exponential distributions so that $\norm{\Phi_{(x,a),:}}/q_1(x, a)$ and $\norm{(P - B)_{:,x}^\top \Phi}/q_2(x)$ are always bounded. 
\item \textbf{One step look-ahead}
  When the columns of $\Phi$ are close to their \textit{one step look-ahead}, there exists a constant $C''>0$ such that for any $x$, $\norm{P_{:,x}^\top \Phi}/\norm{B_{:,x}^\top \Phi} < C''$. If we are also able to compute $Z_1 = \sum_{(x,a)} \norm{\Phi_{(x,a),:}}$ and $Z_2 = \sum_x \norm{B_{:,x}^\top \Phi}$, then it is natural to take $q_1(x,a)= \norm{\Phi_{(x,a),:}}/Z_1$ and $q_2(x)=\norm{B_{:,x}^\top \Phi}/Z_2$.
\end{itemize}

 In what follows, we assume that such distributions $q_1$ and $q_2$ are known.

Minimizing the convex surrogate function does not guarantee a feasible solution to the original dual LP. Therefore, we define the following non-feasibility penalties which roughly correspond to how far $\Phi\theta$ is from the simplex and how far $\Phi\theta$ is from a stationary distribution, respectively:
\begin{align*}
  V_1(\theta) & \df \sum_{(x,a)} \abs{ [\mu_0(x,a) + \Phi_{(x,a),:}\theta]_{-}} \text{ and }\\
  V_2(\theta) & \df \norm{(P-B)^\top(\Phi\theta)}_1
                = \sum_{x'} \abs{(P- B)_{:,x'}^\top \Phi \theta }.
\end{align*}

The rest of the section proves the following theorem, our main guarantee for the stochastic subgradient method.
\begin{theorem}
\label{thm:main_average}
Consider an expanded efficient large-scale dual ALP problem and some error tolerance $\epsilon>0$ and desired maximum probability of error $\delta>0$. Then running the stochastic subgradient method (shown in Figure~\ref{alg:SGD}) with
\[
T\geq \max\left\{\frac{H^2}{\epsilon^2}, 40S^2\log\frac{1}{\delta}\right\}
\quad\text{ and }\quad
  \eta=\left(\sqrt{d} + H (C_1 + C_2)\right)\frac{S}{\sqrt{T}},
\]
yields a $\widehat \theta_T$ where
\begin{align*}
  \ell^\top \mu_{\widehat\theta_T}
    &\leq
      \ell^\top\mu_\theta      
       +2\left(H+ O(1)\right)\left(V_1(\theta) + V_2(\theta)\right)
      + O\left(\frac{1}{H}\right)
      +O\left(\epsilon\right),      
\end{align*}
holds with probability at least $1-\delta$. In particular, for the choice of $H = \epsilon^{-1}$, the bound becomes 
\begin{equation}
  \ell^\top \mu_{\widehat\theta_T}
    \leq
    \ell^\top\mu_\theta 
      +O\left(\frac{1}{\epsilon}\right)\left(V_1(\theta) + V_2(\theta)\right)
      +O\left(\epsilon\right).
      \label{eq:exp-loss}
\end{equation}
Constants hidden in the big-O notation are polynomials in $S$, $d$, $C_1$, $C_2$, $\log(1/\delta)$, $\log(V_1(\theta)+V_2(\theta))$, $\tau(\mu_{\theta})$, and $\tau(\mu_{\widehat\theta_T})$.
\end{theorem}
Functions $V_1$ and $V_2$ are bounded by small constants for any set of normalized features: for any $\theta\in\Theta$,
\begin{align*}
V_1(\theta) &\le \norm{\mu_0}_1 + \norm{\Phi \theta}_1 \le 1 + \sum_{(x,a)} \abs{\Phi_{(x,a),:} \theta} \le 1 + S d \,, \\
V_2(\theta) &\le \sum_{x'} \abs{P_{:,x'}^\top ( \mu_0 + \Phi \theta) } + \sum_{x'} \abs{B_{:,x'}^\top ( \mu_0 + \Phi \theta) } \\
&\le \left(\sum_{x'} P_{:,x'}\right)^\top [ \mu_0 + \Phi \theta ]_+ + \left(\sum_{x'} B_{:,x'}\right)^\top [\mu_0 + \Phi \theta]_+ \\
&= 2 [ \mu_0 + \Phi \theta ]_+^\top\mathbf 1 \\
&\le 2 \abs{ \mu_0 + \Phi \theta }^\top \mathbf 1 \\
&= 2 + 2 S \; .
\end{align*}
Thus $V_1$ and $V_2$ can be very small given a carefully designed set of features. The output $\widehat \theta_T$ is a random vector as the algorithm is based on a stochastic convex optimization method. The above theorem shows that with high probability the policy implied by this output is near optimal.

The optimal choice for $\epsilon$ is $\epsilon=\sqrt{V_1(\theta_*) + V_2(\theta_*)}$, where $\theta_*$ is the minimizer of RHS of \eqref{eq:exp-loss} and not known in advance. One could think of parameterizing the optimization problem by $H$, but the problem is not jointly convex in $H$ and $\theta$. Nevertheless, we present methods that recover a $O(\sqrt{V_1(\theta_*) + V_2(\theta_*)})$ error bound using a grid based method in Section~\ref{sec:finding_H}.

\subsection{Analysis}
  This section provides the necessary technical tools and a proof of the main result. We break the proof into two main ingredients. First, we demonstrate that a good approximation to the surrogate loss gives a feature vector that is almost a stationary distribution; this is Lemma~\ref{lem:V.to.stationary.distribution}.
Second, we justify the use of unbiased gradients in Theorem~\ref{thm:stoch-gradient} and Lemma~\ref{lem:risk-bound}. The section concludes with the proof of Theorem~\ref{thm:main_average}. Long, technical proofs have been moved to Section~\ref{sec:average.cost.proofs} when we felt that their inclusion did not add much insight. 

  The first ingredient shows that we can relate the magnitude of the constraint violation of $\theta$ to the difference between $\Phi\theta$ and $\mu_\theta$, which quantifies how far $\Phi\theta$ is from a stationary distribution.
\begin{lemma}
\label{lem:V.to.stationary.distribution}
Let $u\in \Reals^{\cX \cA}$ be a vector, $\mathcal N$ be the set of points $(x,a)$ where $u (x,a)<0$, and $\mathcal S$ be the complement of $\mathcal N$.
Assume
\[
\sum_{x,a} u (x,a) = 1,\, \sum_{(x,a)\in \mathcal N } \abs{u(x,a)} \le \epsilon',\, \norm{u^\top (P-B)}_1\le \epsilon'' .
\]
The vector $[u]_+/\norm{[u]_+}_1$ defines a policy, which in turn defines a stationary distribution $ \mu_{u}$. We have that
\[
\norm{ \mu_{u} - u}_1 \le \tau(\mu_u) \log(1/\epsilon') (2\epsilon'+\epsilon'') + 3\epsilon' \; .
\]
\end{lemma}

The second ingredient is the validity of the subgradient estimates. We assume access to estimates of the subgradient of a convex cost function. Error bounds can be obtained from results in the stochastic convex optimization literature; the following theorem, a high-probability version of Lemma~3.1 of \citet{Flaxman-Kalai-McMahan-2005} for stochastic convex optimization, is sufficient. We note that the variance reduced stochastic gradient descent literature (e.g.\ SAGA or SVGR) cannot be directly applied since a full gradient calculation is impossible, and most complexity upper bounds are at least $O(\sqrt{\cX\cA}/\epsilon)$ \citep{xiao2014proximal}, which is inappropriate for our setting.

\begin{theorem}
  \label{thm:stoch-gradient}

  Consider a bounded set $\cZ\subset\Reals^d$ of radius $Z$ (i.e.\  $\norm{z}\le Z$ for all $z\in\cZ)$ and a sequence of real-valued convex cost functions $(f_t)_{t=1,2,\dots,T}$. Let $z_1,z_2,\dots,z_T\in \cZ$ be the stochastic gradient decent path defined by defined by $z_1=0$  and $z_{t+1} = \Pi_{\cZ} (z_t - \eta f_t')$, where $\Pi_{\cZ}$ is the Euclidean projection onto $\cZ$, $\eta>0$ is a learning rate, and $f_1',\dots,f_T'$ are bounded unbiased subgradient estimates; that is,  $\ex\left[f_t' | z_t\right] = \nabla f(z_t)$ and $\norm{f_t'}\le F$ for some $F>0$. Then, for $\eta = Z/(F\sqrt{T})$ and any $\delta\in (0,1)$,
\begin{align}
\label{eq:stoch-gradient-excess-loss}
\sum_{t=1}^T &f_t(z_t) - \min_{z\in\cZ} \sum_{t=1}^T f_t(z) \le Z F \sqrt{T} + \sqrt{(1 + 4 Z^2 T ) \left(2 \log\frac{1}{\delta} + d \log \left( 1 + \frac{Z^2 T}{d} \right) \right) } \;
\end{align}
 with probability at least $1-\delta$.
\end{theorem}
\begin{proof}
Let $z_* = \argmin_{z\in\cZ} \sum_{t=1}^T f_t(z)$ and $\eta_t = f_t' - \nabla f_t (z_t)$. Define function $h_t:\cZ\rightarrow\Reals$ by $h_t(z) = f_t(z) + z \eta_t$. Notice that $\nabla h_t(z_t) = \nabla f_t(z_t) + \eta_t = f_t'$. By Theorem~1 of \citet{Zinkevich-2003}, we get that
\[
\sum_{t=1}^T h_t(z_t) - \sum_{t=1}^T h_t(z_*) \le \sum_{t=1}^T h_t(z_t) - \min_{z\in\cZ}\sum_{t=1}^T h_t(z) \le Z F \sqrt{T} \; .
\]
Thus,
\[
\sum_{t=1}^T f_t(z_t) - \sum_{t=1}^T f_t(z_*) \le Z F \sqrt{T} + \sum_{t=1}^T (z_* - z_t) \eta_t \; .
\]
Let $S_t = \sum_{s=1}^{t-1} (z_* - z_s) \eta_s$, which is a self-normalized sum~\citep{delaPena-Lai-Shao-2009}. By Corollary~3.8 and Lemma~E.3 of \citet{Abbasi-Yadkori-2012}, we get that for any $\delta\in (0,1)$, with probability at least $1-\delta$,
\begin{align*}
\abs{S_t} &\le \sqrt{\left(1 + \sum_{s=1}^{t-1} (z_t - z_*)^2 \right) \left(2 \log\frac{1}{\delta} + d \log \left( 1 + \frac{Z^2 t}{d} \right) \right) }  \\
&\le \sqrt{(1 + 4 Z^2 t ) \left(2 \log\frac{1}{\delta} + d \log \left( 1 + \frac{Z^2 t}{d} \right) \right) } \; .
\end{align*}
Thus,
\[
\sum_{t=1}^T f_t(z_t) - \min_{z\in\cZ} \sum_{t=1}^T f_t(z) \le Z F \sqrt{T} + \sqrt{(1 + 4 Z^2 T ) \left(2 \log\frac{1}{\delta} + d \log \left( 1 + \frac{Z^2 T}{d} \right) \right) } \; .
\]
\end{proof}
\begin{remark}
\label{rem:Jensen}
Let $B_T$ denote the RHS of \eqref{eq:stoch-gradient-excess-loss}.
If all cost functions are equal to $f$, then by convexity of $f$ and an application of Jensen's inequality, we obtain that $f(\sum_{t=1}^T z_t/T) - \min_{z\in \cZ} f(z) \le B_T/T$.
\end{remark}

The last step before giving the proof of Theorem~\ref{thm:main_average} is to apply Theorem~\ref{thm:stoch-gradient} to our convex surrogate function, $c(\theta)$.
\begin{lemma}
\label{lem:risk-bound}
Under the same conditions as in Theorem~\ref{thm:main_average} and any $\delta\in (0,1)$
\begin{align}
\label{eq:err-bnd}
  c(\widehat \theta_T) &-  \min_{\theta\in\Theta} c(\theta)
                         \le
                         \frac{S(\sqrt{d} + H (C_1 + C_2))}{\sqrt{T}}
                         + \sqrt{\frac{1 + 4 S^2 T}{T^2} \left(2 \log\frac{1}{\delta}
                         + d \log \left( 1 + \frac{S^2 T}{d} \right) \right) } \;
\end{align}
with probability at least $1-\delta$,
\end{lemma}
The proof (in the appendix) consists of checking that conditions of Theorem~\ref{thm:stoch-gradient} are satisfied

With both ingredients in place, we can prove our main result.
\begin{proof}[Proof of Theorem~\ref{thm:main_average}]
Let $b_T$ be the RHS of \eqref{eq:err-bnd}. Using the trivial fact that $\sqrt{a+b}\leq 2\sqrt{a}+2\sqrt{b}$,
we can easily derive
\begin{align}
  b_T
  \le
  \frac{S}{\sqrt{T}}
  \left(
  (\sqrt{d} + H (C_1 + C_2))
  + 2\sqrt{10 \log\frac{1}{\delta}}
  +2\sqrt{5d \log \left( 1 + \frac{S^2 T}{d} \right)} \right)
  + O\left(\frac{1}{T}\right).
\end{align}

Lemma~\ref{lem:risk-bound} implies that with high probability for any $\theta\in\Theta$,
\begin{align}
\label{eq:online-to-batch}
\ell^\top (\mu_0 + \Phi \widehat \theta_T) + H\, V_1(\widehat\theta_T) + H\, V_2(\widehat\theta_T) \le \ell^\top (\mu_0 + \Phi\theta)+ H\, V_1(\theta)  + H\, V_2(\theta) + b_T\; .
\end{align}
From \eqref{eq:online-to-batch}, we get that
\begin{align}
\label{eq:V1}
V_1(\widehat \theta_T) &\le \frac{1}{H} \left( 2(1+S) + H\, V_1(\theta)  + H\, V_2(\theta)+ b_T \right) \df \epsilon' \,, \\
\label{eq:V2}
V_2(\widehat \theta_T) &\le \frac{1}{H} \left( 2(1+S) + H\, V_1(\theta)  + H\, V_2(\theta) + b_T \right) \df \epsilon'' \; .
\end{align}
Inequalities \eqref{eq:V1} and \eqref{eq:V2} and Lemma~\ref{lem:V.to.stationary.distribution} give the following bound:
\begin{equation}
\abs{ \ell^\top\mu_{\widehat\theta_T} - \ell^\top(\mu_0 + \Phi \widehat\theta_T)} \le
\tau(\mu_{\widehat\theta_T}) \log(1/\epsilon') (2\epsilon'+\epsilon'') + 3\epsilon' \; ,
\end{equation}
and we can similarly bound
\begin{equation}
\abs{ \ell^\top\mu_{\theta} - \ell^\top(\mu_0 + \Phi \theta)} \le
\tau(\mu_{\theta}) \log(1/V_1(\theta)) (2V_1(\theta) + V_2(\theta)) + 3V_1(\theta).
\end{equation}

Combining these two equations with \eqref{eq:online-to-batch} gives the final result:
\begin{align*}
  \ell^\top\mu_{\widehat\theta_T}
  &\le
    \ell^\top(\mu_0 + \Phi\widehat\theta_T) + \tau(\mu_{\widehat\theta_T}) \log(1/\epsilon') (2\epsilon'+\epsilon'') + 3\epsilon'\\
  &\leq
    \ell^\top(\mu_0 + \Phi\theta_T) + \tau(\mu_{\widehat\theta_T}) \log(1/\epsilon') (2\epsilon'+\epsilon'')
    + 3\epsilon' + H V_1(\theta) + H V_2(\theta) + b_T\\
  &\leq
    \ell^\top\mu_\theta +
\tau(\mu_{\theta}) \log(1/V_1(\theta)) (2V_1(\theta) + V_2(\theta)) + 3V_1(\theta)\\
  & \quad + \tau(\mu_{\widehat\theta_T}) \log(1/\epsilon') (2\epsilon'+\epsilon'')
    + 3\epsilon' + H V_1(\theta) + H V_2(\theta) + b_T\\
&\leq
    \ell^\top\mu_\theta +
2\left(V_1(\theta) + V_2(\theta)\right)
\left(3 + \tau(\mu_{\theta}) \log(1/V_1(\theta))
+ \tau(\mu_{\widehat\theta_T}) \log(1/\epsilon')
+H
\right)\\
  & \quad + \left(2\tau(\mu_{\widehat\theta_T}) \log(1/\epsilon')+3\right)\frac{2(1+S)}{H}
+ (2\tau(\mu_{\widehat\theta_T}) \log(1/\epsilon')+3)\frac{b_T}{H}+b_T.
\end{align*}
Using the form of $b_T$ above, we find the excess loss bound
\begin{align}
  \ell^\top\mu_{\widehat\theta_T}
  &\leq
    \ell^\top\mu_\theta +
 2\left(V_1(\theta) + V_2(\theta)\right)
 \left(3 + \tau(\mu_{\theta}) \log(1/V_1(\theta))
+ \tau(\mu_{\widehat\theta_T}) \log(1/\epsilon')
+H \right)\notag\\
 &\quad  + \left(2\tau(\mu_{\widehat\theta_T}) \log(1/\epsilon')+3\right)\frac{2(1+S)}{H}
+ \frac{S}{\sqrt{T}} H (C_1 + C_2)\notag\\
&\quad+ \left(\frac{2\tau(\mu_{\widehat\theta_T}) \log(1/\epsilon')+3}{H}+2\right)
 \frac{S}{\sqrt{T}}
   \sqrt{10 \log\frac{1}{\delta}}\notag\\
   &\quad
  +O\left(\frac{\log(T)}{\sqrt{T}}\right)+ O\left(\frac{1}{\sqrt{T}H}\right)\label{eqn:average.regret.constants}\\
  &\leq
    \ell^\top\mu_\theta +
 2\left(V_1(\theta) + V_2(\theta)\right)\left(H+ O(1)\right)
+O\left(\frac{1}{H}\right)+O\left(\frac{H}{\sqrt{T}}\right)\notag\\
&\quad+
O\left( \frac{1}{H\sqrt{T}}\right)   \sqrt{ \log\frac{1}{\delta}}
  +O\left(\frac{\log(T)}{\sqrt{T}}\right)
\end{align}
Now, recall that we set
\begin{equation*}
 T = \max\left\{\frac{H^2}{\epsilon^2}, 40S^2\log\frac{1}{\delta}\right\},
\end{equation*}
which finally yields that with high probability, for any $\theta\in\Theta$,
\begin{align*}
  \ell^\top \mu_{\widehat\theta_T}
    &\leq
      \ell^\top\mu_\theta      
       +2\left(H+ O(1)\right)\left(V_1(\theta) + V_2(\theta)\right)
      + O\left(\frac{1}{H}\right)
      +O\left(\epsilon\right),      
\end{align*}
as claimed.
\end{proof}

\subsection{Comparison with Previous results}
With a precise statement of our main result, we return to compare Theorem~\ref{thm:main_average} from \citet{DeFarias-VanRoy-2006}. Their approach is to relate the original MDP to a perturbed version \footnote{In a perturbed MDP, the state process restarts with a certain probability to a \textit{restart distribution}. Such perturbed MDPs are closely related to discounted MDPs.} and then analyze the corresponding ALP.  Let $\Psi$ be a feature matrix that is used to estimate value functions. Recall that $\lambda_*$ is the average loss of the optimal policy and $\lambda_{w}$ is the average loss of the greedy policy with respect to value function $\Psi w$. Let $h_\gamma^*$ be the differential value function when the restart probability in the perturbed MDP is $1-\gamma$. For vector $v$ and positive vector $u$, define the weighted maximum norm $\norm{v}_{\infty, u} = \max_x u(x) \abs{v(x)}$. \citet{DeFarias-VanRoy-2006} prove that for appropriate constants $C,C'>0$ and weight vector $u$,
\begin{equation}
\label{eq:deFaVaRo}
\lambda_{w_*} - \lambda_* \le \frac{C}{1-\gamma} \min_{w} \norm{h_\gamma^* - \Psi w}_{\infty, u} + C' (1-\gamma) \; .
\end{equation}
This bound has similarities to bound \eqref{eq:exp-loss}: tightness of both bounds depends on the quality of feature vectors in representing the relevant quantities (stationary distributions in \eqref{eq:exp-loss} and value functions in \eqref{eq:deFaVaRo}). Once again, we emphasize that the algorithm proposed by \citet{DeFarias-VanRoy-2006} is computationally expensive and requires access to a distribution that depends on optimal policy.

\section{Average Cost Meta-Algorithm}
\label{sec:finding_H}
The previous section proved that Algorithm~\ref{alg:SGD} found a $\widehat\theta_T$ with
\begin{equation*}
    \mu_{\widehat\theta_T}^\top \ell \le \min_{\theta\in\Theta} \left(
      \ell^\top\mu_\theta      
       +2\left(H+ O(1)\right)\left(V_1(\theta) + V_2(\theta)\right)
      + O\left(\frac{1}{H}\right)
      +O\left(\epsilon\right) \right),
\end{equation*}
where $H$ is a hyperparameter and $\epsilon$ is some error tolerance. If one has reason to believe that the violation terms $V_1(\theta) + V_2(\theta)$ are negligible (for example, if the features are close to stationary distributions), then one can set $H = \epsilon^{-1}$. However, we wish to be adaptive to the size of the constrain violations around the optimum $\theta^*$, and ideally obtain the excess loss bound 
\begin{equation*}
\ell^\top\mu_{\widehat\theta_T} \le \min_{\theta\in\Theta}  \ell^\top\mu_\theta +O\left( \sqrt{V_1(\theta) + V_2(\theta)}\right) + O(\epsilon),
\end{equation*}
which would imply that we have solved the Expanded Efficient Large-Scale Dual ALP problem (Definition~\ref{defn:E.ELALP}) with violation $V(\theta) = \sqrt{V_1(\theta) + V_2(\theta)}$.

Unfortunately, we must jointly optimize over $\theta$ and $H$ and the objective is not jointly convex. We avoid this difficulty with a \emph{meta-algorithm}, proposed and analyzed in this section.

This meta-algorithm, detailed in Figure~\ref{alg:meta}, uses Algorithm~\ref{alg:SGD} to approximate $\widehat \theta$ over a grid $H_1,\ldots,H_K$ of $H$ values. It takes as inputs a bound on the violation function $V_{\max}$, a desired error tolerance $\epsilon$, and desired probability tolerance $\delta$. The algorithm then carefully chooses a grid $H_1,\ldots, H_K$, and, for each $i=1,\ldots,K$, computes $\hat\theta_i$, the output of Algorithm~\ref{alg:SGD} with parameter $H= H_i$, and $\widehat V_i$, an approximation to $V_1(\hat\theta_i) + V_2(\hat\theta_i)$. It then returns $\hat\theta_{\hat k}$, where
\begin{equation*}
    \hat k \df \argmin_k \ell^\top\Phi\widehat\theta_k + H_k \widehat V_k + \frac{\beta}{H_k}.
\end{equation*}

Intuitively, this two-step procedure approximately computes
\begin{equation}
  \min_{\theta\in\Theta,H\in\Reals} \left( \ell^\top\mu_\theta + H ( V_1(\theta) + V_2(\theta) ) + \frac{\beta}{H} \right),
  \label{eqn:grid_objective}
\end{equation}
which produces a bound that satisfies Definition~\ref{defn:E.ELALP}.

Throughout this section, we use the following notation. We define $c(H,\theta) \df \ell^\top\Phi\theta + H(V_1(\theta)+V_2(\theta))$,  $\theta_H^* \df \argmin_\theta c(H,\theta)$, and $F(H) = c(H,\theta^*_H) + \frac{\beta}{H}$. Hence, the optimization \eqref{eqn:grid_objective} is equal to 
$\min_{H, \theta} c(H,\theta) + \frac{\beta}{H} = \min_{H} F(H)$.

\begin{figure}
\begin{center}
\framebox{\parbox{14cm}{
\begin{algorithmic}
\STATE \textbf{Input: }
Upper bound $V_{\max}$ on $V_1(\theta) + V_2(\theta)$, error tolerance $\epsilon>0$, error probability $\delta>0$, constraint estimation distributions $q_1$ and $q_2$
\STATE Initialize $H_0 \leftarrow \beta\left(\sqrt{V_{\max}}\right)^{-1}$ and $i \leftarrow 0$
\WHILE{ $H_i \leq \frac{2\beta}{\epsilon}$}
\STATE Set
$H_{i+1} \leftarrow H_i + \epsilon \left( V_{\max} +\frac{\beta }{H_i^2} \right)^{-1}$
\STATE Set $i \leftarrow i+1$
\ENDWHILE
\STATE Set $K \leftarrow i$
\FOR{$k=0,1,\ldots,K$}
\STATE Obtain $\widehat\theta_k$ from Algorithm \ref{alg:SGD} with $T =
\max\left\{\frac{H_k^2}{\epsilon^2}, 40S^2\log\frac{K}{\delta}\right\}$
\STATE Set $n \leftarrow \frac{8(S(C_1+1)+SC_2)^2} {\epsilon^2}\log\left(\frac{4K}{\delta}\right)$
\STATE Sample $y_1,\ldots, y_n\sim q_1$ and $(x_1,a_1),\ldots,(x_n,a_n)\sim q_2$
\STATE Set $ \widehat V_k \leftarrow \frac{1}{n}\displaystyle\sum_{i=1}^n  \left[\frac{[\mu_0(x_i,a_i)+\Phi_{(x_i,a_i),:}\hat\theta_k]_-}{q_1(x,a)}
  +\frac{\abs{(P-B)_{:,y_i}^\top \Phi\hat\theta_k}}{q_2(y_i)} \right]$
\ENDFOR
 \STATE Set $\hat k \leftarrow \argmin_k \ell^\top\Phi\hat\theta_k + H_k \widehat V_k + \frac{\beta}{H_k}$
\STATE Return policy $\pi_{\widehat \theta_{\hat k}}$
\end{algorithmic}
}}
\end{center}
\caption{The Meta-algorithm}
\label{alg:meta}
\end{figure}

\subsection{Estimating the Error Functions}
\label{sec:estimate.V.average}
To run the Grid Algorithm, we need to be able to estimate the constraint violations $V_1(\theta)+V_2(\theta)$.
Similar to the gradient estimate, we estimate $V_1+V_2$ by importance-weighted sampling.
For some $n$ and samples $y_1,\ldots, y_n\sim q_1$ and $(x_1,a_1),\ldots,(x_n,a_n)\sim q_2$, define
\begin{equation}
  \widehat V_n(\theta) \df \frac{1}{n}\sum_{i=1}^n  \frac{[\mu_0(x_i,a_i)+\Phi_{(x_i,a_i),:}\theta]_-}{q_1(x_i,a_i)}
  +\frac{\abs{(P-B)_{:,y_i}^\top \Phi\theta}}{q_2(y_i)} .
\end{equation}
Since $  V_1(\theta) =\sum_{(x,a)} \abs{ [\mu_0(x,a) + \Phi_{(x,a),:}\theta]_{-}}$ and
  $V_2(\theta) = \sum_{x'} \abs{(P- B)_{:,x'}^\top \Phi \theta }$, this estimate is clearly unbiased. Also, we earlier assumed the existence of constants
  $C_1 = \max_{(x,a)\in [\cX]\times [\cA]}\frac{\norm{\Phi_{(x,a),:}}}{q_1(x, a)}$
and $C_2 = \max_{x\in [\cX]}\frac{\norm{(P - B)_{:,x}^\top \Phi}}{q_2(x)}$, and so we can bound
\begin{align*}
\frac{[\mu_0(x_i,a_i)+\Phi_{(x_i,a_i),:}\theta]_-}{q_1(x,a)}
  +\frac{\abs{(P-B)_{:,y_i}^\top \Phi\theta}}{q_2(y_i)}
  \leq
  S(C_1+1) + SC_2
\end{align*}
which gives us concentration of $\widehat V$ around $V$. In particular, applying Hoeffding's inequality yields:
\begin{lemma}
  \label{lem:V_hat_error}
  Given $\epsilon>0$ and $\delta \in [0,1]$, for any $\theta$, the violation function estimate $\widehat V_n(\theta)$ has
  \begin{equation*}
    \abs{\widehat V_n(\theta) - (V_1(\theta)+V_2(\theta))} \leq \epsilon
  \end{equation*}
    with probability at least $1-\delta$ as long as we choose
  $
    n \geq  \frac{(S(C_1+1) + SC_2)^2}{2\epsilon^2}\log\left(\frac{2}{\delta}\right).
  $
\end{lemma}

\subsection{Choosing the Coarseness of the Grid}
We wish to construct the sequence $H_k$ such that $\max_{H_k\leq H\leq H_{k+1}} F(H)$ is always $\frac{\epsilon}{2}$, and hence we need control of the smoothness of $F(H)$. Recall that we will choose $H$ to approximately balance the two terms $H V(\theta) + \frac{\beta}{H} \leq H V_{\max} + \frac{\beta}{H}$, and so it suffices to only search for $H \geq \frac{\beta}{\sqrt{V_{\max}}}$. The maximum $H$ will be determined by $\epsilon$. 
\begin{lemma}
\label{lem:H_grid_bound}
Let $\epsilon>0$ be some desired error tolerance and $V_{\max}$ be some upper bound on $V_1(\theta) + V_2(\theta)$; we can always take $V_{\max}= 3 + S(d+2)$. Consider the $H_k$ sequence defined in Algorithm~\ref{alg:meta} by the base case $H_0 \df  \beta\left(\sqrt{V_{\max}}\right)^{-1}$, induction step
$H_{k+1} \df H_k + \epsilon \left( V_{\max} +\frac{\beta }{H_k^2} \right)^{-1}$, and terminal condition
$ K \df \min\left\{i\in\nat : H_i \geq \frac{2\beta}{\epsilon}\right\}$.
The grid $H_0,\ldots, H_K$ has the property that 
\begin{equation}
  \max_{H,H' \in [H_k, H_{k+1}]} \abs{F(H) - F(H')}\leq \epsilon.
\end{equation}
Additionally, we have $K = O(\log( 1/\epsilon))$.
\end{lemma}

\begin{proof}
  Our first goal is to bound $\max_{H, H'\in [H_i, H_{i+1}]} \abs{F(H) - F(H')}$. We first note that $c(H,\theta^*_H)$, which is a function of $H$ only, is increasing since 
\begin{align*}
  c(H,\theta^*_H)
  & =
    \min_{\theta}\ell^\top\Phi\theta + H(V_1(\theta)+V_2(\theta))\\
  &\leq
    \min_{\theta}\ell^\top\Phi\theta + (H+\delta)(V_1(\theta)+V_2(\theta))\\
  &=
  c(H+\delta,\theta^*_{H+\delta}).
\end{align*}
We also note that $c(H,\theta^*_H)$ is sublinear in $H$, and indeed 
\begin{align*}
  c(H+\delta,\theta^*_{H+\delta})
  &=
    \min_\theta \ell^\top\Phi\theta + (H+\delta)(V_1(\theta)+V_2(\theta))\\
  &\leq
    \ell^\top\Phi\theta^*_H + (H+\delta)(V_1(\theta^*_H)+V_2(\theta^*_H))\\
  &=
      c(H,\theta^*_{H}) + \delta(V_1(\theta^*_H)+V_2(\theta^*_H))\\
  &\leq c(H,\theta^*_{H}) + \delta V_{\max}.
\end{align*}
The two observations imply that
\begin{align*}
  \max_{H, H'\in [H_i, H_{i+1}]} \abs{ c(H',\theta^*_{H'}) - c(H,\theta^*_H)}
  \leq
  c(H_i,\theta^*_{H_i}) + V_{\max}\left( H_{i+1} - H_i\right),
\end{align*}
and hence we may bound
\begin{align*}
  \max_{H, H'\in [H_i, H_{i+1}]} \abs{F(H) - F(H')}
  &\leq
   \abs{c(H_{i+1}, \theta^*_{H_{i+1}}) - c(H_{i}, \theta^*_{H_i})}
    + \beta \max_{H_i\leq H\leq H_{i+1}}\abs{\frac{1}{H} - \frac{1}{H'}}\\
  &\leq
    (H_{i+1}-H_i)V_{\max}
    + \beta\left(\frac{1}{H_{i}} - \frac{1}{H_{i+1}}\right).
\end{align*}

We now check that the grid has the property that
\[
  V_{\max}(H_{i+1}-H_i)
  + \beta\left(\frac{1}{H_{i}} - \frac{1}{H_{i+1}}\right) \leq \epsilon
\]
for all $i \geq 0$. Defining $\Delta_i = H_{i+1} - H_i$, we see that $\Delta_i = \epsilon \left( V_{\max} +\frac{\beta }{H_i^2} \right)^{-1}$ for all $i$. The left hand side of the above condition is equal to 
\begin{align*}
  V_{\max}\Delta
  + \beta\left(\frac{1}{H_{i}} - \frac{1}{H_{i}+\Delta}\right)
=
  \Delta \left( V_{\max} +\frac{\beta }{H_i(H_i + \Delta)} \right)
  &\leq
    \Delta \left( V_{\max} +\frac{\beta }{H_i^2} \right)
  = \epsilon,
\end{align*}
giving us the desired condition.

Lastly, we calculate an upper bound on $K$, the number of grid points needed. We can write
\[
  H_{i+1} = H_i\left( 1 +  \frac{\epsilon} {V_{\max}H_i +\frac{\beta }{H_i}} \right),
\]
and using the bounds
$H_K \leq  \frac{2\beta}{\epsilon}$ and $H_i \geq \beta V_{\max}^{-\frac12}$, we have that
$V_{\max}H_i +\frac{\beta }{H_i} \leq 2\beta V_{\max}/\epsilon +\sqrt{V_{\max}}$, which implies that
\[
  H_k \geq H_0 \left( 1 +  \frac{\epsilon}{ 2\beta V_{\max}/\epsilon +\sqrt{V_{\max}}}\right)^k.
\]
Since we defined $K \df \min\left\{i\in\nat : H_i \geq \frac{2\beta}{\epsilon}\right\}$, we conclude that $K > K'$, where $K'$ is the smallest index such that 
\begin{align*}
  &H_0 \left( 1 +  \frac{\epsilon}{ 2\beta V_{\max}/\epsilon +\sqrt{V_{\max}}}\right)^{K'}
  \geq \frac{2\beta}{\epsilon}\\
    \Leftrightarrow \quad&
                     K' 
                           \geq
                           \frac{\log\left(\frac{2\sqrt{V_{\max}}}{\epsilon}\right)}{\log \left( 1 +  \frac{\epsilon}{ 2\beta V_{\max}/\epsilon +\sqrt{V_{\max}}}\right)},                           
  \end{align*}
  leading to the conclusion that $K = O(\log(1/\epsilon))$.
\end{proof}

\subsection{Meta-Algorithm Excess Loss Bound}
Combining the results from the last two section yields the following theorem.
\begin{theorem}
  \label{thm:H_opt_average_cost}
  For some $\epsilon>0$ and $\delta \in [0,1]$, the Meta-Algorithm specified in Figure~\ref{alg:meta} has excess loss
\begin{align}
  \mu_{\widehat\theta_T}^\top \ell
  &\leq
    \mu_{\theta}^\top \ell + O\left(\sqrt{V_1(\theta)+V_2(\theta)}\right)
    +O(\epsilon)
\end{align}
with probability at least $1-\delta$. It requires  $O\left(\epsilon^{-4}\right)$ subgradient steps and $O\left(\epsilon^{-2}\log(1/\delta)\right)$ samples to estimate the constraint violations.
\end{theorem}
In particular, adapting to the optimal $H$ only introduces logarithmic terms to the run time.

\section{The Dual ALP for Discounted Cost}
\label{sec:discounted_cost}
We now change settings to discounted cost and try to find a policy with discounted cost almost as low as the best in the class. Most of the tools from the average cost carry over with small modifications, and we will focus on presenting the results in this section with most of the theorem proofs presented in the appendix. 

Recall that the LP we intend to approximately solve is
\begin{align*}
&\min_{\theta\in\Reals^{d}}~ \ell^\top\Phi\theta\,, \\
&\mbox{s.t.} \quad (B - \gamma P)^\top \Phi\theta  = \alpha,\quad \Phi\theta\geq 0.
\end{align*}
This LP has another interpretation. The dual of the approximate dual is
\begin{align*}
    & \max_{J \in \Reals^{\cX}}  \alpha^\top J  \\
    & \mbox{s.t.}          \Phi^\top \left( \ell + (\gamma P- B) J - z \right) =0, \\
    &                     z \geq 0,
\end{align*}
which can be viewed as the original primal with constraint aggregation.

\paragraph{Approximately solving the LP}
Analogous to $V_1$ and $V_2$, we define, relative to a feature matrix $\Phi$, the constraint violation functions
\begin{align*}
  V_3(\theta) & \df  \|\left[\Phi \theta \right]_{-}\|_1 \quad \text{ and }\\
  V_4(\theta) & \df \|(B-\gamma P)^T \Phi \theta  -\alpha\|
\end{align*}
so that we can approximate the solution of the LP by minimizing the convex surrogate
\begin{align}
\label{eqn:discounted_surrogate}
  c^\gamma(\theta)
  &\df \ell^\top \Phi\theta + H\left( V_3(\theta) + V_4(\theta)\right) \\
  & =
    \ell^\top \Phi \theta  + H \|\left[\Phi \theta \right]_{-}\|_1
    + H  \|(B-\gamma P)^\top \Phi \theta -\alpha \|_1 \notag\\
& =
    \ell^\top \Phi \theta  + H \sum_{(x,a)}\left[\Phi_{(x,a),:} \theta \right]_{-}
    + H\sum_{x'}  \left|(B-\gamma P)_{:,x'}^\top \Phi \theta -\alpha \right| \notag
\end{align}
with some constant $H$ and the constraint set $\Theta=\{\theta: \|\theta\|_2 \leq S\}$ .

We will minimize \eqref{eqn:discounted_surrogate} through stochastic subgradient descent by sampling 
$(x_t,a_t)\sim q_3\in\triangle_{\cX \times \cA}$ and $x_t' \sim q_4\in\triangle_{\cX}$ and calculating the unbiased estimator of the subgradient, 
\begin{align}
\label{eq:grad_discounted}
g^\gamma_t(\theta) &= \ell^\top \Phi - H \frac{\Phi_{(x_t,a_t),:}}{q_3(x_t,a_t)} \one{\Phi_{(x_t,a_t),:} \theta < 0} + H  \frac{(P-\gamma B)_{:,x_t'}^\top \Phi}{q_4(x_t')} \sgn((P-\gamma B)_{:,x_t'}^\top \Phi \theta).
\end{align}
The algorithm for the average cost case is exactly the same as Figure~\ref{alg:SGD} with $g_t^\gamma$ instead of $g_t$.
Recall that we are using the shorthand
\begin{align*}
J_\theta  = J_{\pi_{\Phi\theta}} \quad \text{ and }\quad
\nu_\theta = \nu_{\pi_{\Phi\theta}}.
\end{align*}
Thus, our objective is to show that $\alpha^\top J_{\widehat \theta_T}$ is small.

A key difference between the average and discounted cases is the interpretation for the dual variables, $\mu$ and $\nu$. In the average case, the feasible $\mu$ exactly corresponded to stationary distributions and therefore the average loss was precisely $\ell^\top\mu$. However, in the discounted case, the dual variables $\nu$ correspond to the expected discounted number of visits to each state and $\ell^\top\nu = \alpha^\top J$, where $J$ is the value function corresponding to policy $\pi_\nu$.

\subsection{A Excess Loss Bound for a Fixed $H$}
  Unlike the average cost case, the discounted cost case does not need a fast mixing assumption. Instead, we assume that the operator 1-norm of $\Phi$ is upper bounded by some constant $C$:
\begin{eqnarray}
  \|\Phi\|_1= \max_{x: \|x\|_1=1} \|\Phi x\|_1 =\max \limits_{1 \leq j \leq d} \sum _{(x,a)} | \Phi_{(x,a),j} | \leq C.
\end{eqnarray}
We also need to assume coverage of the constraint sampling distribution, analogously to the average cost case. We assume existence of constants $C_3$ and $C_4$ such that
\begin{align*}
C_3 \geq \max_{(x,a)\in [\cX]\times [\cA]}\frac{\norm{\Phi_{(x,a),:}}}{q_3(x, a)}\,, \qquad C_4 \geq \max_{x\in [\cX]}\frac{\norm{(P - \gamma B)_{:,x}^\top \Phi}}{q_4(x)}.
\end{align*}
Special structure may suggest natural choices of sampling distributions to ensure small $C_3$ and $C_4$. For example, if $P$ is sparse with support on only $N$ elements and if there is $C'>0$ such that for any $(x,a)$ and $i$, $\Phi_{(x,a), i}\le C'/(\cX \cA)$ and each column of $P$ has only $N$ non-zero elements, we can choose $q_3$ and $q_4$ to be uniform distributions and we can bound
\begin{align*}
\frac{\norm{\Phi_{(x,a),:}}}{q_3(x, a)} \le C'\,, \qquad \frac{\norm{(P - \gamma B)_{:,x}^\top \Phi}}{q_4(x)} \le C' (N + \cA) \; .
\end{align*}
Finally, note that we can always upper bound the constraint violation functions. For any $\theta \in \Theta$,
\begin{align*}
V_3(\theta) &\le  \norm{\Phi \theta}_1 \le \sum_{j=1}^d \sum_{(x,a)} \abs{\Phi_{(x,a),j}} |\theta_j| \leq C \|\theta\|_1 \leq C \sqrt{d} \|\theta\|_2 \leq \sqrt{d} \; CS, \text{ and }\\
V_4(\theta) &\le  \sum_{x'} \abs{B_{:,x'}^\top (  \Phi \theta) } +\gamma \sum_{x'} \abs{P_{:,x'}^\top (   \Phi \theta) } + \|\alpha\|_1 \\
&\le  \sum_{(x,a)} \left(\sum_{x'} B_{(x,a),x'} \right) |(\Phi\theta)_{(x,a)}|   +  \gamma \sum_{(x,a)} \left(\sum_{x'} P_{(x,a),x'} \right) |(\Phi\theta)_{(x,a)}|   + 1  \\
& = (1+\gamma)\|\Phi\theta \|_1 +1 \\
&\le  (1+\gamma) \sqrt{d} \; CS  +1.
\end{align*}
We can combine both statements and obtain
\begin{equation}
\label{eqn:V_bound_discounted}
   V_3(\theta) + V_4(\theta)
   \leq
   1+\sqrt{d}CS(2+\gamma)
   \leq
      4\sqrt{d}CS,
  \end{equation}
as long as $C\geq$ and $S\geq 1$.

The method we propose for optimizing $\pi_\theta$ in the discounted cost setting is to apply stochastic subgradient descent (from Figure~\ref{alg:SGD}) to subgradients $g^\gamma(\theta_t)$ defined in \eqref{eq:grad_discounted}. 
Our algorithm for optimizing discounted cost MDPs is just Figure~\ref{alg:SGD} run with subgradient $g^\gamma(\theta_t)$ (defined in \eqref{eq:grad_discounted}) instead of $g(\theta)$.

We now present the excess loss bound for discounted cost and a fixed $H$.
\begin{theorem}
\label{thm:main_discounted}
Consider an expanded efficient large-scale dual ALP problem and some error tolerance $\epsilon>0$, desired maximum probability of error $\delta>0$, and parameter $H\geq 1$. Running the stochastic subgradient method (Figure~\ref{alg:SGD} with $g^\gamma(\theta)$) with 
\begin{equation}
  \label{eq:discounted_min_T}
  T =
  \frac{S^2}{\epsilon^2}
  \left(H(C_3+C_4)+\sqrt{d}+2\sqrt{10 \log\frac{1}{\delta}}+2\sqrt{5d \log \left( 1 + \frac{S^2 T}{d} \right)}\right)^2
\end{equation}

and constant learning rate $\eta=S/(G'\sqrt{T})$, where $G' = \sqrt{d} + H (C_3 + C_4)$, yields a $\widehat \theta_T$ with
\begin{align*}
  \ell^{\top} \nu_{\widehat\theta_T}
  &\leq
    \ell^{\top} \nu_{\theta} +
    \left( \frac{6}{1-\gamma} + H \right)
    \left(V_3(\theta)+V_4(\theta)\right)
    +\frac{6\sqrt{d} CS}{H(1-\gamma)}+O(\epsilon).
\end{align*}
Constants hidden in the big-O notation are polynomials in $S$, $d$, $C_3$, $C_4$, and $C$.
\end{theorem}
Because the proof is very similar to the average cost section, it has been deferred to  Section~\ref{sec:main_discounted_proof}.

\subsection{Error Bound}
Previous ADP literature concentrated on showing that the optimal value is well approximated if the feature space contains elements close to the optimum; i.e.\  $| \alpha^\top J_{\widehat\theta_T} -\alpha^\top J^*|$ was bounded in terms of $\min_{\theta} \|\Phi\theta-\nu^*\|_1$. Theorem~\ref{thm:main_discounted} is certainly more general, as it remains non-trivial even if
$\min_{\theta} \|\Phi\theta-\nu^*\|_1$ is large. For completeness, we provide a corollary of this form.

\begin{corollary}
\label{cor:discounted_error_bound}
Under the same conditions as Theorem~\ref{thm:main_discounted},
\begin{equation}
  \alpha^\top J_{\widehat\theta_T} -\alpha^\top J^* \leq  C_3
  \left( \frac{1}{1-\gamma} + \frac{1}{\epsilon} \right)
  \min_{\theta}  \|\Phi \theta   -\nu^* \|_1 + C_2 \frac{\epsilon}{1-\gamma}.
\end{equation}
\end{corollary}
\begin{proof}[Proof of Corollary~\ref{cor:discounted_error_bound}]
Let $\theta^*$  be one of the vectors minimizing
$ \|\Phi\theta-\nu^*\|_1$. Theorem~\ref{thm:main_discounted} gives
\begin{eqnarray*}
  \alpha^\top J_{\theta_T} - \alpha^\top J_{\theta^*} \leq C_1
\left( \frac{1}{1-\gamma} + \frac{1}{\epsilon} \right)
\left(V_3(\theta^*)+V_4(\theta^*)\right) + C_2 \frac{\epsilon}{1-\gamma},
\end{eqnarray*}
Since $\nu^* \geq 0$ and by the simple fact that
$[x]_{-} \leq |y-x|$ for any $y \geq 0$, we have
\begin{eqnarray}
  V_3(\theta^*) \leq \|\Phi\theta-\nu^*\|_1.
  \label{eq:V_3_star}
\end{eqnarray}
For the term $V_4(\theta^*)$, since $\nu^*$ is feasible (i.e., $(B-\gamma P)^\top \nu^*=\alpha$)
\begin{align}
  V_4(\theta^*) &\leq \|(B-\gamma P)^\top  (\Phi \theta^* -\nu^*) \|_1 + \|(B-\gamma P)^\top \nu^*-\alpha\|_1 =\|(B-\gamma P)^\top  (\Phi \theta^* -\nu^*) \|_1 \notag\\
                & \leq \|(B-\gamma P)^\top \|_{1} \|\Phi\theta-\nu^*\|_1 \leq \left( \|B^\top \|_{1} +\gamma \|P^\top \|_{1} \right) \|\Phi\theta-\nu^*\|_1  \notag\\
                & =(1 + \gamma) \|\Phi\theta-\nu^*\|_1,
  \label{eq:V_4_star}
\end{align}
where $\|\cdot\|_1$ is the matrix operator 1-norm.
Therefore, we have,
\begin{equation*}
  \alpha^\top J_{\pi_{ [\Phi \widehat\theta_T]_{+}}} - \alpha^\top J_{\pi_{ [\Phi \theta^*]_{+}}} \leq C_1  \left( \frac{1}{1-\gamma} + \frac{1}{\epsilon} \right)
  (2 + \gamma) \|\Phi \theta^*  -\nu^*  \|_1 + C_2 \frac{\epsilon}{1-\gamma}.
  \label{eq:gap_opt_theta}
\end{equation*}
Next, we bound $\alpha^\top J_{\pi_{ [\Phi \theta^*]_{+}}} - \alpha^\top J^*$. Since $\alpha^\top J_{\pi_{ [\Phi \theta^*]_{+}}}= \ell^{\top} \nu_{\pi_{ [\Phi  \theta^*]_{+}}} $ and $\alpha^\top J^*= \ell^{\top} \nu^* $ and by Lemma~\ref{lemma:key_discounted},
\begin{align*}
  \alpha^\top J_{\pi_{ [\Phi \theta^*]_{+}}} - \alpha^\top J^* & \leq \|\ell\|_{\infty} \|\nu_{\pi_{ [\Phi  \theta^*]_{+}}} - \nu^*\|_1 \leq  \|\nu_{\pi_{ [\Phi  \theta^*]_{+}}} - \Phi \theta^*\|_1 + \|\Phi \theta^* - \nu^*\|_1 \\
  & \leq \frac{3 V_3(\theta^*) + V_4(\theta^*)}{1-\gamma}  + \|\Phi \theta^* - \nu^*\|_1  \leq  \frac{5}{1-\gamma}  \|\Phi \theta^* - \nu^*\|_1.
  \label{eq:gap_opt_mu}
\end{align*}
where the last inequality is due to \eqref{eq:V_3_star} and \eqref{eq:V_4_star}. The theorem statement follows from combining these two results.
\end{proof}

\subsection{The Meta-Algorithm for Discounted Cost}
\label{sec:discounted_meta}

\begin{figure}
\begin{center}
\framebox{\parbox{14cm}{
\begin{algorithmic}
\STATE \textbf{Input: }
Upper bound $V_{\max}$ on $V_3(\theta) + V_4(\theta)$, error tolerance $\epsilon>0$, error probability $\delta>0$, constraint estimation distributions $q_3$ and $q_4$
\STATE Initialize $H_0 \leftarrow \beta\left(\sqrt{V_{\max}}\right)^{-1}$ and $i \leftarrow 0$
\WHILE{ $H_i \leq \frac{2\beta}{\epsilon}$}
\STATE Set
$H_{i+1} \leftarrow H_i + \epsilon \left( V_{\max} +\frac{\beta }{H_i^2} \right)^{-1}$
\STATE Set $i \leftarrow i+1$
\ENDWHILE
\STATE $K \leftarrow i$
\FOR{$k=0,1,\ldots,K$}
\STATE Obtain $\widehat\theta_k$ from Algorithm \ref{alg:SGD} with
$T= O\left(H_k^2S^2\log\left(\frac{1}{\delta}\right)\right)$ set by \eqref{eq:discounted_min_T}
\STATE Set $n \leftarrow
     \frac{(S(C_3 + 2C_4))^2}{2\epsilon^2}\log\left(\frac{4K}{\delta}\right)$
\STATE Sample $y_1,\ldots, y_n\sim q_3$ and $(x_1,a_1),\ldots,(x_n,a_n)\sim q_4$
\STATE Set $ \widehat V_k \leftarrow \frac{1}{n}\displaystyle\sum_{i=1}^n  \left[\frac{[\mu_0(x_i,a_i)+\Phi_{(x_i,a_i),:}\hat\theta_k]_-}{q_3(x,a)}
  +\frac{\abs{(P-\gamma B)_{:,y_i}^\top \Phi\hat\theta_k}}{q_4(y_i)} \right]$
\ENDFOR
 \STATE Set $\hat k \leftarrow \argmin_k \ell^\top\Phi\hat\theta_k + \left(H_k+\frac{1}{1-\gamma}\right) \widehat V_k + \frac{\beta}{H_k(1-\gamma)}$
\STATE Return policy $\pi_{\widehat \theta_{\hat k}}$
\end{algorithmic}
}}
\end{center}
\caption{The Meta-algorithm for Discounted Cost}
\label{alg:meta_discounted}
\end{figure}

Analogously to the average cost case, setting $H$ correctly yields a excess loss bound of $O\left(\sqrt{V_3(\theta^*) + V_3(\theta^*)}\right) + O(\epsilon)$. The excess loss bound from Theorem~\ref{thm:main_discounted} suggests that we want $H$ and $\theta$ to optimize
\[
    \ell^{\top} \Phi\theta
    +\left( \frac{6}{1-\gamma} + H \right)
    \left(V_3(\theta)+V_4(\theta)\right)
    +\frac{\beta}{H},
\]
where we have defined $\beta\df \frac{6\sqrt{d}CS}{(1-\gamma)}$. The Meta-Algorithm for discounted cost, presented in Figure~\ref{alg:meta_discounted}, operates in a manner very similar to the average cost case: a grid $H_1,\ldots, H_K$ is chosen, the corresponding $\widehat\theta_k$ are computer, then $\pi_{\widehat\theta_{\hat k}}$, where 
\begin{equation*}
    \hat k \df \argmin_k \ell^\top\Phi\widehat\theta_k + \left(H_k+\frac{1}{1-\gamma}\right) \widehat V_k + \frac{\beta}{H_k},
  \end{equation*}
is returned.  We can prove the following bound for the meta-algorithm.
\begin{theorem}  
  \label{thm:H_opt_discounted_cost}
  For some $\epsilon>0$ and $\delta \in [0,1]$, the Meta-Algorithm for discounted cost (Figure~\ref{alg:meta_discounted} has excess loss
\begin{equation*}
  \ell^\top\nu_{\theta_{\hat k}} \leq \min_\theta \ell^\top\nu_\theta
  + O\left(\sqrt{V_3(\theta)+V_4(\theta)}\right)
  + O(\epsilon),
\end{equation*}
with probability at least $1-\delta$.
It requires  $O\left(\epsilon^{-4}\right)$ subgradient steps and $O\left(\epsilon^{-2}\log(1/\delta)\right)$ samples to estimate the constraint violations.
\end{theorem}
For the proof and technical details, please see Section~\ref{sec:meta.analysis.discounted}.

\section{Experiments}
\label{sec:experiments}

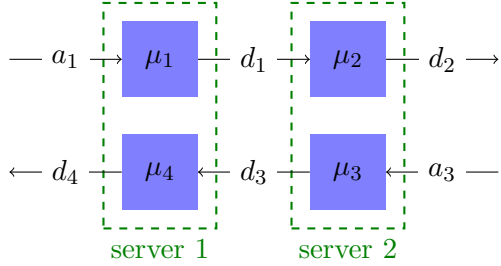
\begin{figure}[t]
\centering
\begin{tikzpicture}
  \node[rectangle, fill=blue!50!white, minimum width=1cm, minimum height=1cm] (q1) at (2, 2.5) {$\mu_1$};
  \node[rectangle, fill=blue!50!white, minimum width=1cm, minimum height=1cm] (q2) at (4.5, 2.5) {$\mu_2$};

  \node[rectangle, fill=blue!50!white, minimum width=1cm, minimum height=1cm] (q4) at (2, 1) {$\mu_4$};
  \node[rectangle, fill=blue!50!white, minimum width=1cm, minimum height=1cm] (q3) at (4.5, 1) {$\mu_3$};

  \draw[thick, dashed, green!50!black] (1.25, .25) node[below,xshift = .75cm] {server 1} rectangle(2.75,3.25);
  \draw[thick, dashed, green!50!black] (3.75, .25) node[below,xshift = .75cm] {server 2} rectangle(5.25,3.25);

  \draw[->] (q1) to [edge node={node [sloped,fill=white] {$d_1$}}]  (q2);
  \draw[->] (0,2.5) to [edge node={node [sloped,fill=white] {$a_1$}}]  (q1);
  \draw[->] (q2) to [edge node={node [sloped,fill=white] {$d_2$}}]  (6.5,2.5);

  \draw[->] (q3) to [edge node={node [sloped,fill=white] {$d_3$}}]  (q4);
  \draw[->] (q4) to [edge node={node [sloped,fill=white] {$d_4$}}]  (0,1);
  \draw[->] (6.5,1) to [edge node={node [sloped,fill=white] {$a_3$}}]  (q3);

\end{tikzpicture}
\caption{
\label{fig:4Dqueue}
The 4D queuing network. Customers arrive at queue $\mu_1$ or $\mu_3$ then are referred to queue $\mu_2$ or $\mu_4$, respectively. Server 1 can either process queue 1 or 4, and server 2 can only process queue 2 or 3.}
\end{figure}

In this section, we apply both algorithms to the four-dimensional discrete-time queuing network illustrated in Figure \ref{fig:4Dqueue}. This network has a relatively long history; see, e.g. \cite{Rybko-Stolyar-1992} and more recently \cite{DeFarias-VanRoy-2003} (c.f. Section 6.2). There are four queues, $\mu_1,\ldots,\mu_4$, each with state $0,\ldots,B$. Since the cardinality of the state space is $\cX=(1+B)^4$, even a modest $B$ results in huge state-spaces. For time $t$, let $X_t\in [\cX]$ be the state and $s_{i,t}\in\{0,1\}$, $i=1, 2, 3, 4$ denote whether queue $i$ is being served. Server 1 only serves queue 1 or 4, server 2 only serves queue 2 or 3, and neither server can idle. Thus, $s_{1,t}+s_{4,t}=1$ and $s_{2,t}+s_{3,t}=1$. The dynamics are as follows. At each time $t$, the following random variables are sampled independently: $A_{1,t}\sim\text{Bernoulli}(a_1)$, $A_{3,t}\sim\text{Bernoulli}(a_3)$, and $D_{i,t}\sim\text{Bernoulli}(d_i s_{i,t})$ for $i=1,2,3,4$. Using $e_1,\ldots,e_4$ to denote the standard basis vectors, the dynamics are:
\begin{align*}
X'_{t+1}=&X_t+A_{1,t}e_1+A_{3,t}e_3
+D_{1,t}(e_2-e_1)-D_{2,t}e_2
+D_{3,t}(e_4-e_3)-D_{4,t}e_4,
\end{align*}
and $X_{t+1}=\max(\mathbf{0},\min(\mathbf B,X'_{t+1}))$ (i.e. all four states are thresholded from below by 0 and above by $B$).
The loss function is the total queue size: $\ell(X_t)=||X_t||_1$. We compared our method against two common heuristics. In the first, denoted LONGER, each server operates on the queue that is longer with ties broken uniformly at random (e.g. if queue 1 and 4 had the same size, they are equally likely to be served). In the second, denoted LBFS (last buffer first served), the downstream queues always have priority (server 1 will serve queue 4 unless it has length 0, and server 2 will serve queue 2 unless it has length 0). These heuristics are common and have been used as benchmarks for queuing networks (e.g. \cite{DeFarias-VanRoy-2003}).

We used $a_1=a_3=.08$, $d_1=d_2=.12$, and $d_3=d_4=.28$, and buffer sizes $B_1=B_4=38$, $B_2=B_3=25$ as the parameters of the network.. The asymmetric size was chosen because server 1 is the bottleneck and tend to have longer queues. The first two features are features of the stationary distributions corresponding to two heuristics.  We also included two types of non-stationary-distribution features. For every interval $(0,5],(6,10],\ldots,(45,50]$ and action $A$, we added a feature $\psi$ with $\phi(x,a)=1$ if $\ell(x,a)$ is in the interval and $a=A$. To define the second type, consider the three intervals $I_1=[0,10]$, $I_2=[11, 20]$, and $I_3=[21, 25]$. For every 4-tuple of intervals $(J_1,J_2,J_3,J_4)\in\{I_1,I_2,I_3\}^4$ and action $A$, we created a feature $\psi$ with $\psi(x,a)=1$ only if $x_i\in J_i$ and $a=A$. Every feature was normalized to sum to 1. In total, we had 372 features which is about a $10^{4}$ reduction in dimension from the original problem.

\begin{figure*}[t]
\centering
\includegraphics[width=16cm]{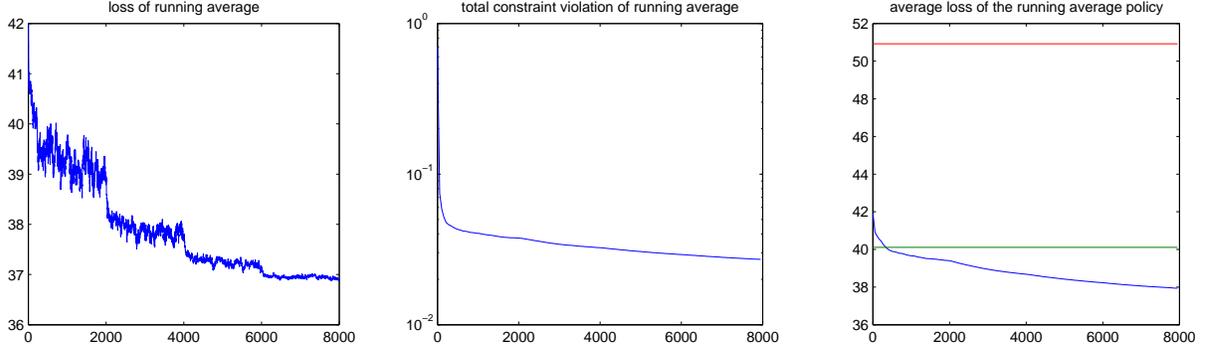}
\caption{
\label{fig:4Dplots}
The left plot is of the linear objective of the running average, i.e. $\ell^\top \Phi\widehat\theta_t$. The center plot is the sum of the two constraint violations of $\widehat\theta_t$, and the right plot is $\ell^\top \tilde\mu_{\widehat\theta_t}$ (the average loss of the derived policy). The two horizontal lines correspond to the loss of the two heuristics, LONGER and LBFS.}
\end{figure*}

We ran our stochastic subgradient descent algorithm with $I=1000$ sampled constraints and constraint gain $H=2$. Our learning rate began at $10^{-4}$ and halved every $2000$ iterations. The results of our algorithm are plotted in Figure \ref{fig:4Dplots}, where $\widehat\theta_t$ denotes the running average of $\theta_t$. The left plot is of the LP objective, $\ell^\top(\mu_0+\Phi\widehat\theta_t)$. The middle plot is of the sum of the constraint violations, $\norm{[\mu_0+\Phi\widehat\theta_t]_{-} }_1+ \norm{(P-B)^\top\Phi\widehat\theta_t }_1$. Thus, $c(\widehat\theta_t)$ is a scaled sum of the first two plots. Finally, the right plot is of the average losses, $\ell^\top\mu_{\widehat\theta_t}$ and the two horizontal lines correspond to the loss of the two heuristics, LONGER and LBFS. The right plot demonstrates that, as predicted by our theory, minimizing the surrogate loss $c(\theta)$ does lead to lower average losses.

All previous algorithms (including \cite{DeFarias-VanRoy-2003}) work with value functions, while our algorithm works with stationary distributions. Due to this difference, we cannot use the same feature vectors to make a direct comparison. The solution that we find in this different approximating set is slightly worse than the solution of \citet{DeFarias-VanRoy-2003}.

\section{Conclusion}
This paper demonstrated the feasibility of solving the MDP planning problem with a parametric policy class based on an approximate dual LP. Unlike previous approaches, we were able to prove \emph{excess loss bounds}, that is, bounds relative to the best policy in our parametric class. We obtained results for both the average cost and discounted cost settings as well as empirical justification.

There are several promising directions. First, are such excess loss bounds possible in the primal formulation?

Another drawback to our methods is that we need a backwards simulator, that is, access to every state with positive probability of transitioning into a state $x$. Are there alternative formulations that remove this requirement?

\bibliography{all_bib}

\section{Acknowledgments}
We gratefully acknowledge the support of the NSF through grant CCF-1115788 and of the ARC through an Australian Research Council Australian Laureate Fellowship (FL110100281).

\appendix

\section{Deferred Proofs for Average Cost}
\label{sec:average.cost.proofs}
\begin{proof}[Proof of Lemma~\ref{lem:V.to.stationary.distribution}]
Let $f = u^\top (P-B)$. From $\norm{u^\top (P-B)}_1 \le \epsilon''$, we get that for any $x'\in [\cX]$,
\begin{align*}
\sum_{(x,a)\in \mathcal S} & u(x,a) (P-B)_{(x,a), x'} = - \sum_{(x,a)\in \mathcal N} u(x,a) (P-B)_{(x,a), x'} + f(x')
\end{align*}
such that $\sum_{x'} \abs{f(x')} \le \epsilon''$. Let $h = [u]_+/\norm{[u]_+}_1$. Let $H' = \norm{h^\top (B-P)}_1$. We write
\begin{align*}
H' &= \sum_{x'} \abs{\sum_{(x,a)\in \mathcal S} h(x,a)  (B-P)_{(x,a),x'} } \\
&=  \frac{1}{1+\epsilon'} \sum_{x'} \abs{\sum_{(x,a)\in \mathcal S} u(x,a)  (B-P)_{(x,a),x'} } \\
&= \frac{1}{1+\epsilon'}  \sum_{x'} \abs{-\sum_{(x,a)\in \mathcal N} u(x,a) (B-P)_{(x,a),x'} + f(x')} \\
&\le \frac{1}{1+\epsilon'} \left( \sum_{x'} \abs{-\sum_{(x,a)\in \mathcal N} u(x,a) (B-P)_{(x,a),x'}} + \sum_{x'} \abs{f(x')} \right) \\
&\le \frac{1}{1+\epsilon'} \left( \epsilon'' + \sum_{(x,a)\in \mathcal N} \sum_{x'} \abs{u(x,a)} \abs{(B-P)_{(x,a),x'}} \right)  \\
&\le \frac{1}{1+\epsilon'} \left( \epsilon'' + \sum_{(x,a)\in \mathcal N} 2 \abs{u(x,a)} \right) \le \frac{2\epsilon' + \epsilon''}{1+\epsilon'} \\
&\le 2 \epsilon' + \epsilon''  \; .
\end{align*}
Vector $h$ is almost a stationary distribution in the sense that
\begin{equation}
\label{eq:almost-stationary}
\norm{h^\top (B-P)}_1 \le 2 \epsilon' + \epsilon''\; .
\end{equation}
Let $\norm{w}_{1,\mathcal S} = \sum_{(x,a)\in \mathcal S} \abs{w(x,a)}$. First, we have that
\begin{align*}
\norm{h - u}_1 &\le \norm{h - \frac{u}{1+\epsilon'}}_1 + \norm{u - \frac{u}{1+\epsilon'}}_{1,\mathcal S} \le 2 \epsilon' \; .
\end{align*}
Next we bound $\norm{ \mu_{h} - h}_1$. Using $\nu_0 = h$ as the initial state distribution, we will show that as we run policy $h$ (equivalently, policy $ \mu_{h}$), the state distribution converges to $ \mu_{h}$ and this vector is close to $h$. From \eqref{eq:almost-stationary}, we have $\mu_0^\top P = h^\top B + v_0$, where $v_0$ is such that $\norm{v_0}_1\le 2\epsilon'+\epsilon''$. Let $M^{h}$ be a $\cX\times (\cX\cA)$ matrix that encodes policy $h$, $M_{(i,(i-1)\cA+1)\mbox{-}(i,iA)}^{h}=h(\cdot|x_i)$. Other entries of this matrix are zero. We have
\begin{align*}
h^\top P M^{h} &= (h^\top B + v_0) M^{h} = h^\top B M^{h} + v_0 M^{h}= h^\top + v_0 M^{h} \,,
 \end{align*}
where we used the fact that $h^\top B M^{h} = h^\top$. Let $\mu_1^\top = h^\top P M^{h}$ which is the state-action distribution after running policy $h$ for one step. Let $v_1 = v_0 M^{h} P = v_0 P^{h}$ and notice that as $\norm{v_0}_1\le 2\epsilon'+\epsilon''$, we also have that $\norm{v_1}_1 = \norm{P^{h \top} v_0^\top}_1 \le \norm{v_0}_1\le 2\epsilon'+\epsilon''$. Thus,
\[
\mu_1^\top P = h^\top P + v_1 =  h^\top B + v_0 + v_1  \; .
\]
By repeating this argument for $k$ rounds, we obtain
\[
\mu_k^\top = h^\top + (v_0 + v_1 + \dots + v_{k-1}) M^{h}
\]
and it is easy to see that
\[\norm{(v_0 + v_1 + \dots + v_{k-1}) M^{h}}_1 \le \sum_{i=0}^{k-1} \norm{v_i}_1 \le k( 2\epsilon'+\epsilon'').
\] Thus, $\norm{\mu_k - h}_1 \le k (2 \epsilon'+\epsilon'')$. Now, notice that $\mu_k$ is the state-action distribution after $k$ rounds of policy $ \mu_{h}$. By the mixing assumption, $\norm{\mu_k -  \mu_{h}}_1 \le e^{-k/\tau(h)}$, so the choice of $k=\tau(h) \log (1/\epsilon')$ yields $\norm{ \mu_{h} - h}_1 \le \tau(h) \log(1/\epsilon') (2\epsilon'+\epsilon'') + \epsilon'$.

\end{proof}

\begin{proof}[Proof of Lemma~\ref{lem:risk-bound}]
  We prove the lemma by showing that conditions of Theorem~\ref{thm:stoch-gradient} are satisfied. The assumptions allow an easy bound on the subgradient estimate:
\[
\norm{g_t} \le \norm{\ell^\top \Phi} + H \frac{\norm{\Phi_{(x_t,a_t),:}} }{q_1(x_t,a_t)} + H  \frac{\norm{(P-B)_{:,x_t'}^\top \Phi}}{q_2(x_t')} \le \sqrt{d} + H (C_1 + C_2) \; .
\]

Also, we show that the subgradient estimate is unbiased:
\begin{align*}
\ex\left[g_t(\theta\right] &= \ell^\top \Phi - H \sum_{(x,a)} q_1(x,a) \frac{\Phi_{(x,a),:}}{q_1(x,a)} \one{\mu_0(x,a)+\Phi_{(x,a),:} \theta < 0} \\
&\qquad\qquad\qquad\qquad+ H \sum_{x'} q_2(x') \frac{(P-B)_{:,x'}^\top \Phi}{q_2(x')} \sgn((P-B)_{:,x'}^\top \Phi \theta) \\
&= \ell^\top \Phi - H \sum_{(x,a)} \Phi_{(x,a),:} \one{\mu_0(x,a)+\Phi_{(x,a),:} \theta < 0}+ H \sum_{x'}  (P-B)_{:,x'}^\top \Phi \sgn((P-B)_{:,x'}^\top \Phi \theta) \\
&= \nabla_{\theta} c(\theta) \; .
\end{align*}
The result then follows from Theorem~\ref{thm:stoch-gradient} and Remark~\ref{rem:Jensen}.

  It is also convenient to bound the norm of the gradient.
If $\mu_0(x,a)+\Phi_{(x,a),:} \theta \ge 0$, then $\nabla_{\theta} \abs{[\mu_0(x,a)+\Phi_{(x,a),:} \theta]_{-}} = 0$. Otherwise, $\nabla_{\theta} \abs{[\mu_0(x,a)+\Phi_{(x,a),:} \theta]_{-}} = -\Phi_{(x,a),:}$.
Calculating,
\begin{equation}\label{eq:cost_gradient}
\begin{split}
\nabla_{\theta} c(\theta) &= \ell^\top \Phi + H \sum_{(x,a)} \nabla_{\theta} \abs{[\mu_0(x,a)+\Phi_{(x,a),:} \theta]_{-}}+ H \sum_{x'}\nabla_{\theta} \abs{ (P-B)_{:,x'}^\top \Phi \theta}  \\
&= \ell^\top \Phi - H \sum_{(x,a)} \Phi_{(x,a),:} \one{\mu_0(x,a)+\Phi_{(x,a),:} \theta < 0}+ H \sum_{x'}  (P-B)_{:,x'}^\top \Phi \sgn((P-B)_{:,x'}^\top \Phi \theta)   \,,
\end{split}
\end{equation}
where $\sgn(z) = \one{z>0} - \one{z<0} $ is the sign function. Let $\pm$ denote the plus or minus sign (the exact sign does not matter here). We have that
\begin{align*}
\norm{\nabla_{\theta} c(\theta)} &\le H \sqrt{\sum_{i=1}^d \left(\sum_{x'} \left(\pm\sum_{(x,a)} (P-B)_{(x,a),x'} \Phi_{(x,a),i} \right) \right)^2 } +\norm{ \ell^\top \Phi} + H \sqrt{\sum_{i=1}^d \left(\sum_{(x,a)} \abs{\Phi_{(x,a),i}} \right)^2 } \; .
\end{align*}
Thus,
\begin{align*}
\norm{\nabla_{\theta} c(\theta)} &\le \sqrt{\sum_{i=1}^d (\ell^\top \Phi_{:,i})^2} + H \sqrt{d} + H \sqrt{\sum_{i=1}^d \left(\sum_{(x,a)} \left(\pm\sum_{x'} (P-B)_{(x,a),x'}  \right) \Phi_{(x,a),i} \right)^2 } \\
&\le \sqrt{d} + H \sqrt{d}  + H \sqrt{\sum_{i=1}^d \left(2 \sum_{(x,a)} \abs{\Phi_{(x,a),i}} \right)^2 } =\sqrt{d} (1 + 3 H )  \,,
\end{align*}
where we used $\abs{ \ell^\top \Phi_{:,i}} \le \norm{\ell}_\infty \norm{\Phi_{:,i}}_1 \le 1$.
\end{proof}

\begin{proof}[Proof of Theorem~\ref{thm:H_opt_average_cost}]
  By Theorem~\ref{thm:main_average}, running Algorithm~\ref{alg:SGD} for a given $H_k$ with $T_k = \max\left\{ 16\frac{H_k^2}{\epsilon^2}, 160 S^2 \log\left(\frac{2K}{\delta}\right)\right \}$ produces a $\widehat \theta_k$ with
\[
  c(H_k,\widehat\theta_k)\leq c(H_k,\theta_k^*) + H_kV(\theta^*)+ \frac{\beta}{H_k} + \frac{\epsilon}{4},
\]
where $\theta^*_k = \min_\theta C(H_k,\theta)$, and the probability of error for any single $\widehat\theta_k$ is guaranteed to be at most $\frac{\delta}{2K}$. Hence, the union bound implies that the total probability of error of any $\widehat\theta_k$ is at most $\frac{\delta}{2}$. Similarly, with our choice of $n = \frac{8(S(C_1+1)+SC_2)^2} {\epsilon^2}\log\left(\frac{4K}{\delta}\right)$, Lemma~\ref{lem:V_hat_error} guarantees that  $\abs{V_1(\widehat\theta_k)+V_2(\widehat\theta_k) -
  \widehat V_k}\leq\frac{\epsilon}{4}$
holds for all $k$ simultaneously with probability at least $1-\frac{\delta}{2}$

With these two observations, we can bound the suboptimality of the objective. Recalling that $\hat k$ is the minimizer of $\ell^\top\Phi\widehat\theta_{k} + H_{k} \widehat V_k + \frac{\beta}{H_{k}}$, and using $k^*$ as the minimizer of
$c(H_k,\theta_k^*) + \frac{\beta}{H_k}$, we have
    \begin{align*}
      \ell^\top\Phi\widehat\theta_{\hat k} + H_{\hat k} \widehat V_{\hat k}
      + \frac{\beta}{H_{\hat k}}
      &=
        \min_k \ell^\top\Phi\widehat\theta_{k} + H_{k} \widehat V_k
      + \frac{\beta}{H_{k}}\\
      &\leq
        \ell^\top\Phi\widehat\theta_{k^*} + H_{k^*} \widehat V_{k^*}
        + \frac{\beta}{H_{k^*}}\\
      &\leq
        \ell^\top\Phi\widehat\theta_{k^*} + H_{k^*} (V_1(\hat\theta_{k^*}) + V_2(\hat\theta_{k^*}))
         + \frac{\beta}{H_{k^*}} + \frac{\epsilon}{4}
      &\text{(Lemma~\ref{lem:V_hat_error})}\\
      &=
        c(H_{k^*},\widehat \theta_{k^*})
         + \frac{\beta}{H_{k^*}} + \frac{\epsilon}{4}\\
      &\leq
        c(H_{k^*},\theta_{k^*}^*) + \frac{\beta}{H_{k^*}}
        + \frac{\epsilon}{2}\\
      &=\min_k
        c(H_k,\theta_k^*) + \frac{\beta}{H_k}
        + \frac{\epsilon}{2}\\
      &\leq  \min_{H,\theta} c(H,\theta) + \frac{\beta}{H}
        + \epsilon.
      &\text{(Lemma~\ref{lem:H_grid_bound})}
\end{align*}
One final application of the union bound guarantees that the statement holds with probability $1-(\frac{\delta}{2} + \frac{\delta}{2})$. Hence, the Meta-algorithm minimizes the objective to within $\epsilon$.

We next relate the suboptimality of the objective optimization to the suboptimality of the true loss $\ell^\top \mu_{\theta_{\hat k}}$. Since all quantities are non-negative, this implies that $\abs{\frac{\beta}{H_{\hat k}} - \frac{\beta}{H^*}}\leq\epsilon$. Finally, we can put together the excess loss bound. To apply Lemma 3 and bound the distance between
$\ell^\top\Phi\mu_{\widehat\theta_{\hat k}}$ and
$\ell^\top\Phi\widehat\theta_{\hat k}$, we first need to bound $V_1(\widehat\theta_{\hat k})$ and $V_2(\widehat\theta_{\hat k})$. Using the bounded suboptimality of $\widehat\theta_{\hat k}$ as an optimizer of $c(H_{\hat k},\theta)$, we have
\begin{align*}
  \ell^\top\Phi\widehat\theta_{\hat k}
  + H_{\hat k}\left(V_1(\widehat\theta_{\hat k}) + V_2(\widehat\theta_{\hat k})\right)
  &\leq
  \ell^\top\Phi\theta^*_{\hat k}
  + H_{\hat k}\left(V_1(\theta^*_{\hat k}) + V_2(\theta^*_{\hat k})\right)+ \frac{\epsilon}{2}\\
  & \leq
  \ell^\top\Phi\theta^*
  + H^*\left(V_1(\theta^*) + V_2(\theta^*)\right)+ \epsilon
\end{align*}
and can conclude that
\begin{align*}
  V_1(\widehat\theta_{\hat k})
  &\leq
    \frac{1}{H_{\hat k}}\left(
    2(S+1) + \sqrt{V_1(\theta^*) + V_2(\theta^*)}\right)\\
  &\leq
    \left(\frac{1}{H^*}+\epsilon\right)
    \left(
    2(S+1) + \sqrt{V_1(\theta^*) + V_2(\theta^*)}
    \right)\\
  &=
     (2(S+1)+\epsilon)  \sqrt{V_1(\theta^*) + V_2(\theta^*)}
    + (V_1(\theta^*) + V_2(\theta^*))
    + 2(S+1)\epsilon.
\end{align*}
Completely analogous reasoning gives the same bound on  $V_2(\widehat\theta_{\hat k})$.

Then, applying Lemma~\ref{lem:V.to.stationary.distribution}, we have
\begin{align*}
\ell^\top\Phi\mu_{\theta_{\hat k}}
  &\leq
    \ell^\top\Phi\widehat\theta_{\hat k}
    + 4\tau(\mu_{\theta_{\hat k}})\log(1/\epsilon')
    \left(
    (2(S+1)+\epsilon)  \sqrt{V_1(\theta^*) + V_2(\theta^*)}
    + (V_1(\theta^*) + V_2(\theta^*))
    + 2(S+1)\epsilon
    \right)\\
  &\leq
    \ell^\top\Phi\widehat\theta^*
    + 4\tau(\mu_{\theta_{\hat k}})\log(1/\epsilon')
    \left(
    (2(S+1)+\epsilon)  \sqrt{V_1(\theta^*) + V_2(\theta^*)}
    + (V_1(\theta^*) + V_2(\theta^*))
    + 2(S+1)\epsilon
    \right)\\
  &\quad
    + H^* (V_1(\theta^*) + V_2(\theta^*)) + \frac{\beta}{H^*} + \epsilon\\
  &\leq
    \ell^\top\mu_{\theta^*}
    + 4\tau(\mu_{\theta_{\hat k}})\log(1/\epsilon')
    \left(
    (2(S+1)+\epsilon)  \sqrt{V_1(\theta^*) + V_2(\theta^*)}
    + (V_1(\theta^*) + V_2(\theta^*))
    + 2(S+1)\epsilon
    \right)\\
  &\quad
    + H^* (V_1(\theta^*) + V_2(\theta^*)) + \frac{\beta}{H^*} + \epsilon
    + (V_1(\theta^*) + V_2(\theta^*)).
\end{align*}
Plugging in $H^* = \left(\sqrt{V_1(\theta)+V_2(\theta)}\right)^{-1}$ produces
\begin{equation*}
  \ell^\top\mu_{\theta_{\hat k}} \leq \min_\theta \ell^\top\mu_\theta + O\left(\sqrt{V_1(\theta)+V_2(\theta)}\right)
  + O\left(V_1(\theta)+V_2(\theta)\right) + O(\epsilon).
\end{equation*}
The theorem statement follows by recalling that $V_1(\theta)+ V_2(\theta) \leq 1$.

Let us turn to the complexity. The total number of subgradient descent steps is bounded by
\[
  K T_{K} = 16\frac{2\beta^2}{\epsilon^4}                      \frac{\log\left(\frac{2\sqrt{V_{\max}}}{\epsilon}\right)}{\log \left( 1 +  \frac{\epsilon}{ 2\beta V_{\max}/\epsilon +\sqrt{V_{\max}}}\right)}
  = O\left(\epsilon^{-4}\right)
\]
and the total number of samples needed to estimate the violation function is
\[
  nK
  =
  \frac{8(S(C_1+1)+SC_2)^2} {\epsilon^2}\log\left(\frac{4K}{\delta}\right)
  \frac{\log\left(\frac{2\sqrt{V_{\max}}}{\epsilon}\right)}{\log \left( 1 +  \frac{\epsilon}{ 2\beta V_{\max}/\epsilon +\sqrt{V_{\max}}}\right)}
  =
  O\left(\epsilon^{-2} \log(1/\delta)\right).
\]
\end{proof}

\section{Discounted Cost Excess Loss Analysis}
\label{sec:main_discounted_proof}
This section presents the necessary technical tools and the proof of Theorem~\ref{thm:main_discounted}. We begin by showing that if some vector $\nu$ is close to a feasible point of the LP, then it almost equals the expected frequencies of visits of the policy $\pi_\nu$ (when the system runs under the policy $\pi_{h}$ with the initial distribution $\alpha$), i.e.,
  \begin{eqnarray}
    \nu_{\pi_\nu}(x,a) =\sum_{x'} \alpha(x') \sum_{t=1}^{\infty} \gamma^{t-1} P^{\pi_{h}}\left(x_t=x, a_t=a| x_1 =x' \right).
    \label{eq:mu_pi_h}
  \end{eqnarray}
\begin{lemma}
  \label{lemma:key_discounted}
  For any vector $\nu \in \mathbb{R}^{\cX\cA}$, let $\mathcal{N}$ be the set of points $(x,a)$ where $\nu(x,a)\leq 0$ and $\mathcal{S}=\mathcal{N}^c$ and define the constants $\sum_{(x,a) \in \mathcal{N}} | \nu(x,a)| = \epsilon'$ and $\|(B-\gamma P)^\top \nu -\alpha\|_1=\epsilon''$. Further assume that for each $x$, there exists an $a$ such that $(x,a) \in \mathcal{S}$. Then, for the policy $\pi_{\nu}$ define by
   \begin{eqnarray}
    \pi_{\nu}(a|x)
     = \frac{[\nu(x,a)]_{+}}{\sum_{a'} [\nu(x,a')]_{+}},
    \label{eq:pi_nu}
   \end{eqnarray}
the expected frequencies of visits under the policy is close to $\nu$:
  \begin{eqnarray*}
    \|\nu_{\pi_\nu}- \nu\|_{1} \leq \frac{3 \epsilon' + \epsilon''}{1-\gamma}.
  \end{eqnarray*}
\end{lemma}
\begin{proof}
  First, we notice that,
  \begin{eqnarray}
    \|\left[\nu\right]_+-\nu\|_1 \leq \sum_{(x,a) \in \mathcal{N}} |\nu(x,a)|_1 = \epsilon'.
    \label{eq:h_mu_diff}
  \end{eqnarray}

  Let $\xi = (B-\gamma P)^\top \nu -\alpha \in \mathbb{R}^\cX$ with $\|\xi\|_1 = \epsilon''$ according to the assumption. For any $x' \in [\cX]$, we have,
  \begin{eqnarray*}
    \sum_{(x,a) \in \mathcal{S}} \nu(x,a) (B-\gamma P)_{(x,a), x'} -\alpha(x') =- \sum_{(x,a) \in \mathcal{N}} \nu(x,a) (B-\gamma P)_{(x,a), x'}  + \xi(x').
  \end{eqnarray*}
  Let $v_0=(B-\gamma P)^\top h - \alpha$, we have
  \begin{align}
    \|v_0\|_1 & = \sum_{x'} \left| \sum_{(x,a)} h(x,a) (B-\gamma P)_{(x,a), x'} -\alpha(x')  \right| \nonumber \\
       & = \sum_{x'} \left| \sum_{(x,a) \in \mathcal{S}} \nu(x,a) (B-\gamma P)_{(x,a), x'} -\alpha(x')  \right| \nonumber \\
       & = \sum_{x'} \left| - \sum_{(x,a) \in \mathcal{N}} \nu(x,a) (B-\gamma P)_{(x,a), x'} + \xi(x')  \right|
  \end{align}
  with the upper bound
\begin{align}
\|v_0\|_1 & \leq
 \sum_{x'} \left| - \sum_{(x,a) \in \mathcal{N}} \nu(x,a) (B-\gamma P)_{(x,a), x'} \right| + \|\xi\|_1 \nonumber \\
       & \leq \sum_{(x,a) \in \mathcal{N}}\left( |\nu(x,a)|  \sum_{x'} \left| (B-\gamma P)_{(x,a), x'} \right| \right) + \epsilon'' \nonumber \\
       & \leq 2 \sum_{(x,a) \in \mathcal{N}} |\nu(x,a)| + \epsilon'' \nonumber \\
       & \leq 2 \epsilon' + \epsilon''.
       \label{eq:v_0}
  \end{align}

   Let $M^h$ be a $\cX \times (\cX\cA)$ matrix that encodes the policy $\pi_\nu$, where $
   M^h_{(i, (i-1)\cA+1)- (i, iA)} = \pi_\nu \left(\cdot | x_i \right).$
  As a concrete example with state space $\{x_1, x_2\}$ and action space $\{a_1, a_2\}$, we have
  \begin{eqnarray*}
    M^h= \begin{pmatrix}
      \pi_\nu(a_1 |x_1) & \pi_\nu(a_2 | x_1) & 0 & 0 \\
       0 & 0 & \pi_\nu(a_1 |x_2) & \pi_\nu(a_2 | x_2) \\
    \end{pmatrix}.
  \end{eqnarray*}
  By the definition of $\pi_\nu$ in \eqref{eq:pi_nu}, it is easy to check that $h^\top B M^h= h^\top$.

  With $M^h$, the $\nu_{\pi_{h}}$ defined in \eqref{eq:mu_pi_h} can be written as,
  \begin{eqnarray}
    \nu_{\pi_{h}}^\top=\sum_{t=1}^{\infty}\gamma^{t-1}   \alpha^\top M^h (PM^h)^{t-1}
    \label{eq:mu_pi_h_matrix}
  \end{eqnarray}

  Now, we are ready to bound $\|\nu_{\pi_\nu} - \nu\|_1$. By the definition of $v_0$ (i.e., $v_0=(B-\gamma P)^\top h - \alpha$), we have,
  \begin{eqnarray*}
    \alpha^\top M^h = h^\top B M^h - \gamma h^\top P M^h - v_0^\top M^h=h^\top- \gamma h^\top P M^h - v_0^\top M^h,
  \end{eqnarray*}
  where the last equality is due to $h^\top B M^h= h^\top$. Therefore,
  \begin{eqnarray*}
     \alpha^\top M^h (PM)^{t-1} = h^\top  (PM^h)^{t-1} - \gamma  h^\top (P M^h)^t - v_0^\top M^h (PM)^{t-1}  ,
  \end{eqnarray*}
  By \eqref{eq:mu_pi_h_matrix}, we have,
  \begin{eqnarray}
     \nu_{\pi_{h}}^\top=h^\top - \sum_{t=1}^{\infty} \gamma^{t-1} v_0^\top M_h (PM^h)^{t-1}.
     \label{eq:mu_pi_h_matrix_1}
  \end{eqnarray}

  Let $z_{t}=v_0^\top M_h (PM^h)^{t}$.  By \eqref{eq:v_0}, we have
  \begin{eqnarray*}
    \|z_0\|=\|v_0^\top M_h\|_1 =\sum_{x,a} |v_0(x) \pi_\nu(a|x)| \leq \sum_{x} \left( |v_0(x)| \sum_a |\pi_\nu(a|x)|\right)= \|v_0\|_1 \leq 2 \epsilon' + \epsilon''.
  \end{eqnarray*}
  Further,
  \begin{align*}
    \|z_{t+1}\|_1 =\|z_t P M^h\|_1 &= \sum_{x,a}\sum_{x',a'} \left| z_t(x',a') P(x|x',a') \pi_\nu(a|x) \right| \\
      &\leq \sum_{x,a}\left(  \left| z_t(x',a')\right|  \sum_{x',a'} \left| P_{\pi_\nu} (x, a|x',a') \right| \right) = \|z_t\|_1.
  \end{align*}
  By the induction, we know that $\|z_t\|_1 \leq 2 \epsilon' + \epsilon''$ for all $t$. By \eqref{eq:mu_pi_h_matrix_1},
  \begin{eqnarray}
    \|\nu_{\pi_{h}}-h\|_1 \leq \sum_{t=1}^{\infty} \gamma^{t-1} \|z_{t-1}\|_1 \leq \frac{2 \epsilon' + \epsilon''}{1-\gamma}.
  \end{eqnarray}
  Combining this with \eqref{eq:h_mu_diff} and the triangle inequality,
  \begin{eqnarray}
    \|\nu_{\pi_{h}}- \nu \|_1 \leq \frac{2 \epsilon' + \epsilon''}{1-\gamma} +\epsilon' \leq \frac{3 \epsilon' + \epsilon''}{1-\gamma}.
  \end{eqnarray}
\end{proof}

Next, we need the analog of Lemma~\ref{lem:risk-bound} for the discounted case, which is again a direct application of Theorem~\ref{thm:stoch-gradient}.
\begin{lemma}
\label{lem:risk-bound-discounted}
  Given some error tolerance $\epsilon>0$ and desired maximum probability of error $\delta>0$, running the stochastic subgradient method (shown in Figure~\ref{alg:SGD}) on $c^\gamma(\theta)$ with $T\geq 1/\epsilon^4$, $H=1/\epsilon$, and constant learning rate $\eta=\frac{S}{\sqrt{T}}\left(\sqrt{d} + H (C_3 + C_4)\right)$
produces a $\widehat\theta_T$ such that, with probability at least $1-\delta$,
\begin{align}
\label{eqn:discounted_b_T}
c^\gamma(\widehat \theta_T) &-  \min_{\theta\in\Theta} c^\gamma(\theta) \le
S\frac{\sqrt{d} + H (C_3 + C_4)}{\sqrt{T}} + \sqrt{\frac{1 + 4 S^2 T}{T^2} \left(2 \log\frac{1}{\delta} + d \log \left( 1 + \frac{S^2 T}{d} \right) \right) } \; .
\end{align}
\end{lemma}
\begin{proof}
We (once again) prove the lemma by showing that conditions of Theorem~\ref{thm:stoch-gradient} are satisfied. First, the subgradient norms have the easy bound
\[
\norm{g_t^\gamma} \le \norm{\ell^\top \Phi} + H \frac{\norm{\Phi_{(x_t,a_t),:}} }{q_3(x_t,a_t)} + H  \frac{\norm{(P-\gamma B)_{:,x_t'}^\top \Phi}}{q_4(x_t')} \le \sqrt{d} + H (C_3 + C_4) \; .
\]

Finally, we show that the subgradient estimate is unbiased:
\begin{align*}
\ex\left[g_t^\gamma(\theta)\right] &= \ell^\top \Phi - H \sum_{(x,a)} q_3(x,a) \frac{\Phi_{(x,a),:}}{q_3(x,a)} \one{\mu_0(x,a)+\Phi_{(x,a),:} \theta < 0} \\
&\qquad\qquad\qquad\qquad+ H \sum_{x'} q_4(x') \frac{(P-\gamma B)_{:,x'}^\top \Phi}{q_4(x')} \sgn((P-\gamma B)_{:,x'}^\top \Phi \theta) \\
&= \ell^\top \Phi - H \sum_{(x,a)} \Phi_{(x,a),:} \one{\mu_0(x,a)+\Phi_{(x,a),:} \theta < 0}+ H \sum_{x'}  (P-\gamma B)_{:,x'}^\top \Phi \sgn((P-\gamma B)_{:,x'}^\top \Phi \theta) \\
&= \nabla_{\theta} c^\gamma(\theta) \; .
\end{align*}
\end{proof}

With this lemma in hand, the proof of Theorem~\ref{thm:main_discounted}] proceeds in much the same way as the proof of Theorem~\ref{thm:main_average}].
\begin{proof}[Proof of Theorem~\ref{thm:main_discounted}]
Recall that the convex surrogate for the discounted cost is
\begin{equation*}
  c^\gamma(\theta)= \ell^\top \Phi \theta  + H \|\left[\Phi \theta \right]_{-}\|_1 + H  \|(B-\gamma P)^\top \Phi \theta -\alpha \|_1,
\end{equation*}
with the constraint set $\Theta=\{\theta: \|\theta\|_2 \leq S\}$.

Now, obtain $\widehat\theta_T$ from the stochastic subgradient descent algorithm.  By Lemma~\ref{lem:risk-bound-discounted}, the error bound must be less than
\begin{equation*}
  b_T
  =
  \frac{S}{\sqrt{T}}
  \left((\sqrt{d} + H (C_3 + C_4))
  + 2\sqrt{10 \log\frac{1}{\delta}}
  +2\sqrt{5d \log \left( 1 + \frac{S^2 T}{d} \right)} \right)
  + O\left(\frac{1}{T}\right).
\end{equation*}
Then with high probability, we have for any $\theta \in \Theta$,
\begin{equation*}
\label{eq:online-to-batch2}
\ell^\top \Phi \widehat \theta_T + H\, V_3(\widehat\theta_T) + H\, V_4(\widehat\theta_T) \le \ell^\top\Phi\theta+ H\, V_3(\theta)  + H\, V_4(\theta) + b_T\; .
\end{equation*}

Since we can bound
\begin{equation*}
  \ell^\top\Phi \theta \leq \|\ell\|_{\infty}\| \Phi \theta\|_1 \leq \sqrt{d} \; CS,
\end{equation*}
rearranging Equation~\eqref{eq:online-to-batch2} yields
\begin{align*}
V_3(\widehat \theta_T) &\le \frac{1}{H} \left( 2\sqrt{d} \; CS + H\, V_3(\theta)  + H\, V_4(\theta)+ b_T \right) \df \epsilon' \,\text{, and } \\
V_4(\widehat \theta_T) &\le \frac{1}{H} \left( 2\sqrt{d} \; CS+ H\, V_3(\theta)  + H\, V_4(\theta) + b_T \right) \df \epsilon'' \; .
\end{align*}
Using these bounds on $V_3(\widehat\theta_T)$ and $V_3(\widehat\theta_T)$ with Lemma~\ref{lemma:key_discounted} gives
\begin{equation*}
  \abs{  \ell^{\top} \nu_{\widehat\theta_T} - \ell^\top \Phi \widehat\theta_T } \le
  \| \nu_{\widehat\theta_T} -\Phi \widehat\theta_T\|_1 \leq \frac{3 \epsilon' + \epsilon''}{1-\gamma}.
\end{equation*}
Lemma~\ref{lemma:key_discounted}, applied to $\nu_\theta$, implies that
\begin{equation*}
  \abs{  \ell^{\top} \nu_{ \theta} - \ell^\top \Phi \theta } \le
  \| \nu_{\theta} -\Phi  \theta\|_1 \leq \frac{3 V_3(\theta) + V_4(\theta) }{1-\gamma},
\end{equation*}
and so
\begin{align*}
\ell^{\top} \nu_{\widehat\theta_T} &\leq  \ell^\top \Phi \widehat\theta_T + \frac{3 \epsilon' + \epsilon''}{1-\gamma} \\
& \leq \ell^\top\Phi\theta + H\, V_3(\theta)  + H\, V_4(\theta) + b_T  + \frac{3 \epsilon' + \epsilon''}{1-\gamma}  \\
&\le  \ell^{\top} \nu_{\theta} +\frac{3 V_3(\theta) + V_4(\theta) }{1-\gamma}  + H\, V_3(\theta)  + H\, V_4(\theta) + b_T  + \frac{3 \epsilon' + \epsilon''}{1-\gamma}\; .
\end{align*}
First, we simplify
\begin{align*}
    \frac{3\epsilon' + \epsilon''}{1-\gamma}
  &=
    \frac{3}{H(1-\gamma)} \left( 2\sqrt{d}CS+ H V_3(\theta)  + H V_4(\theta) + b_T \right)\\
  &=
    \frac{3}{(1-\gamma)}\left( V_3(\theta)  + V_4(\theta)\right)
    + \frac{3}{H(1-\gamma)}2\sqrt{d}CS
    +\frac{4S(\sqrt{d}+C_3+C_4)}{\sqrt{T}H(1-\gamma)}\\
  &\quad + \frac{3S}{\sqrt{T}H(1-\gamma)}2\sqrt{10 \log\frac{1}{\delta}}
    +\frac{3S}{\sqrt{T}H(1-\gamma)}2\sqrt{5d \log \left( 1 + \frac{S^2 T}{d} \right)}
    +O\left(\frac{1}{T^{3/2}(1-\gamma)H}\right)\\
  &=
    \frac{3}{(1-\gamma)}\left( V_3(\theta)  + V_4(\theta)\right)
    + \frac{6}{H(1-\gamma)}\sqrt{d}CS
    +O\left(\frac{\log(T)}{(1-\gamma)H\sqrt{T}}\right).
\end{align*}
Plugging in this expression and $b_T$, we have
\begin{align*}
  \ell^{\top} \nu_{\widehat\theta_T}
  &\leq
    \ell^{\top} \nu_{\theta} +
    \left( \frac{6}{1-\gamma} + H \right)
    \left(V_3(\theta)+V_4(\theta)\right)
    +\frac{6\sqrt{d} CS}{H(1-\gamma)}
    +O\left(\frac{\log(T)}{(1-\gamma)H\sqrt{T}}\right)+b_T\\
  &\leq
    \ell^{\top} \nu_{\theta} +
    \left( \frac{6}{1-\gamma} + H \right)
    \left(V_3(\theta)+V_4(\theta)\right)
    +\frac{6\sqrt{d} CS}{H(1-\gamma)}
    +\frac{S}{\sqrt{T}}H (C_3 + C_4) \\
  &\quad
    +\frac{S}{\sqrt{T}}\left(C_3+C_4+\sqrt{d}+2\sqrt{10 \log\frac{1}{\delta}}+2\sqrt{5d \log \left( 1 + \frac{S^2 T}{d} \right)}\right)\\
   &\quad +O\left(\frac{\log(T)}{(1-\gamma)H\sqrt{T}}\right)
    +O\left(\frac{1}{T}\right).
\end{align*}
Thus, setting $T$ such that
\begin{align*}
  T
  &\geq
    \frac{S^2}{\epsilon^2} \left(H(C_3+C_4)+\sqrt{d}+2\sqrt{10 \log\frac{1}{\delta}}+2\sqrt{5d \log \left( 1 + \frac{S^2 T}{d} \right)}\right)^2
\end{align*}
or, more compactly,
$T= O\left(S^2\log\left(\frac{1}{\delta}\right)\frac{H^2}{\epsilon^2}\right)$,
yields
\begin{align*}
  \ell^{\top} \nu_{\widehat\theta_T}
  &\leq
    \ell^{\top} \nu_{\theta} +
    \left( \frac{6}{1-\gamma} + H \right)
    \left(V_3(\theta)+V_4(\theta)\right)
    +\frac{6\sqrt{d} CS}{H(1-\gamma)}
    +O\left(\epsilon\right)
\end{align*}
where, as usual, the $O$ hides log factors. This statement holds
with probability at least $1-\delta$ and for any $\theta \in \Theta$.

\end{proof}

\section{Analysis of the Discounted Cost Meta-Algorithm}
\label{sec:meta.analysis.discounted}
It is important to note that the optimum $H^*$ need never be smaller than $\beta/\sqrt{V_{\max}}$, where $V_{\max}$ is some bound on $V_3(\theta^*)+ V_4(\theta^*)$. Even though we cannot compute this quantity, we may still restrict the domain of $H$ to
\begin{equation*}
   H \geq \min_\theta 1/\sqrt{V_3(\theta) + V_4(\theta)}
   \geq
   \left(1+\sqrt{d}CS(2+\gamma)\right)^{-\frac12}
   \geq
      \left(4\sqrt{d}CS\right)^{-\frac12}
      .
  \end{equation*}
where the bound on $V_3(\theta) + V_4(\theta)$ is taken from \eqref{eqn:V_bound_discounted}.

For convenience, we will overload the notation from the  average cost analysis. Define
\[
  c(H,\theta) \df \ell^\top\Phi\theta + \left(H+\frac{6}{1-\gamma}\right)(V_3(\theta)+V_4(\theta)),
\]
where $\theta_H^* \df \argmin_\theta c(H,\theta)$, and $F(H) = c(H,\theta^*_H) + \frac{\beta}{H}$.
The \emph{meta-algorithm for discounted cost} takes as inputs a bound on the violation function $V_{\max}$, discount factor $\gamma$, an error tolerance $\epsilon$, and desired probability tolerance $\delta$. The algorithm then carefully chooses a grid $H_1,\ldots, H_K$, computes the corresponding $\widehat\theta_k$, then returns $\pi_{\widehat\theta_{\hat k}}$ where 
\begin{equation*}
    \hat k \df \argmin_k \ell^\top\Phi\widehat\theta_k + \left(H_k+\frac{1}{1-\gamma}\right) \widehat V_k + \frac{\beta}{H_k}.
  \end{equation*}

\subsection{Estimating the Violation Functions}
\label{sec:estimate.V}
Given some $\theta$, we can estimate the violation function $V_3(\theta) + V_4(\theta)$ in much the same way as the average cost case. For some $n$ and samples $y_1,\ldots, y_n\sim q_3$ and $(x_1,a_1),\ldots,(x_n,a_n)\sim q_4$, define
\begin{equation}
  \widehat V_n(\theta) \df \frac{1}{n}\sum_{i=1}^n  \frac{[\Phi_{(x_i,a_i),:}\theta]_-}{q_3(x,a)}
  +\frac{\abs{(B-\gamma P)_{:,y_i}^\top \Phi\theta-\alpha}}{q_4(y_i)} .
\end{equation}
Since $ V_3(\theta) =\sum_{(x,a)} \abs{ [\Phi_{(x,a),:}\theta]_{-}}$ and
$V_4(\theta) = \sum_{x'} \abs{(B- \gamma P)_{:,x'}^\top \Phi \theta -\alpha}$, this estimate is clearly unbiased. Also, we earlier assumed the existence of constants
\begin{align*}
C_3 = \max_{(x,a)\in [\cX]\times [\cA]}\frac{\norm{\Phi_{(x,a),:}}}{q_3(x, a)}\,, \qquad C_4 = \max_{x\in [\cX]}\frac{\norm{(P - \gamma B)_{:,x}^\top \Phi}}{q_4(x)} \; ,
\end{align*}
and so we can bound
\begin{align*}
  \frac{[\Phi_{(x_i,a_i),:}\theta]_-}{q_3(x,a)}
  +\frac{\abs{(B-\gamma P)_{:,y_i}^\top \Phi\theta-\alpha}}{q_4(y_i)}
  \leq
  S(C_3 + 2C_4).
\end{align*}
Therefore, we have concentration of $\widehat V$ around $V$. The analogous result to Lemma~\ref{lem:V_hat_error} (also using Hoeffding's inequality) is the following.
\begin{lemma}
  \label{lem:V_hat_error_discounted}
  Given $\epsilon>0$ and $\delta \in [0,1]$, for any $\theta$, the violation function estimate $\widehat V_n(\theta)$ has
  \begin{equation*}
    \abs{\widehat V_n(\theta) - (V_3(\theta) + V_4(\theta))} \leq \epsilon
  \end{equation*}
    with probability at least $1-\delta$ as long as we choose
  $
    n \geq  \frac{(S(C_3 + 2C_4))^2}{2\epsilon^2}\log\left(\frac{2}{\delta}\right).
  $
\end{lemma}

\subsection{Defining the Grid}
As before, let $\epsilon>0$ be some desired error tolerance and $V_{\max}$ be some upper bound on $V_3(\theta) + V_4(\theta)$; we can always take $V_{\max}= 4\sqrt{d}CS$. As we shall see, the $H_k$ sequence can be taken to be identical to the average cost case as long an appropriate $\beta$ and $V_{\max}$ are used. Recall that $H$ is chosen to approximately minimize $\left(H+\frac{1}{\gamma}\right) V(\theta) + \frac{\beta}{H} \leq \left(H+\frac{1}{\gamma}\right) V_{\max} + \frac{\beta}{H}$, and so limiting $H$ to $H \leq \frac{\beta}{\sqrt{V_{\max}}}$ suffices in the discounted case as well.

\begin{lemma}
\label{lem:H_grid_bound_discounted}
Let $\epsilon>0$ be some desired error tolerance and $V_{\max}$ be some upper bound on $V_3(\theta) + V_4(\theta)$; we can always take $V_{\max}= 3 + S(d+2)$. Consider the $H_k$ sequence defined in Algorithm~\ref{alg:meta} by the base case $H_0 \df  \beta\left(\sqrt{V_{\max}}\right)^{-1}$, induction step
$H_{k+1} \df H_k + \epsilon \left( V_{\max} +\frac{\beta }{H_k^2} \right)^{-1}$, and terminal condition
$ K \df \min\left\{i\in\nat : H_i \geq \frac{2\beta}{\epsilon}\right\}$.
The grid $H_0,\ldots, H_K$ has the property that 
\begin{equation}
  \max_{H,H' \in [H_k, H_{k+1}]} \abs{F(H) - F(H')}\leq \epsilon.
\end{equation}
Additionally, we have $K = O(\log( 1/\epsilon))$.
\end{lemma}

\begin{proof}
  Our first goal is to bound $\max_{H, H'\in [H_i, H_{i+1}]} \abs{F(H) - F(H')}$. We first note that $c(H,\theta^*_H)$, which is a function of $H$ only, is increasing since 
\begin{align*}
  c(H,\theta^*_H)
  & =
    \min_{\theta}\ell^\top\Phi\theta + \left(H+\frac{1}{1-\gamma}\right)(V_3(\theta)+V_4(\theta))\\
  &\leq
    \min_{\theta}\ell^\top\Phi\theta + \left(H+\frac{1}{1-\gamma}+\delta\right)(V_3(\theta)+V_4(\theta))\\
  &=
  c(H+\delta,\theta^*_{H+\delta}).
\end{align*}
We also note that $c(H,\theta^*_H)$ is sublinear in $H$, and indeed 
\begin{align*}
  c(H+\delta,\theta^*_{H+\delta})
  &=
    \min_\theta \ell^\top\Phi\theta + \left(H+\frac{1}{1-\gamma}+\delta\right)(V_3(\theta)+V_4(\theta))\\
  &\leq
    \ell^\top\Phi\theta^*_H + \left(H+\frac{1}{1-\gamma}+\delta\right)(V_3(\theta^*_H)+V_4(\theta^*_H))\\
  &=
      c(H,\theta^*_{H}) + \delta(V_3(\theta^*_H)+V_4(\theta^*_H))\\
  &\leq c(H,\theta^*_{H}) + \delta V_{\max}.
\end{align*}
The two observations imply that
\begin{align*}
  \max_{H, H'\in [H_i, H_{i+1}]} \abs{ c(H',\theta^*_{H'}) - c(H,\theta^*_H)}
  \leq
  c(H_i,\theta^*_{H_i}) + V_{\max}\left( H_{i+1} - H_i\right),
\end{align*}
and hence we may bound
\begin{align*}
  \max_{H, H'\in [H_i, H_{i+1}]} \abs{F(H) - F(H')}
  &\leq
   \abs{c(H_{i+1}, \theta^*_{H_{i+1}}) - c(H_{i}, \theta^*_{H_i})}
    + \beta \max_{H_i\leq H\leq H_{i+1}}\abs{\frac{1}{H} - \frac{1}{H'}}\\
  &\leq
    (H_{i+1}-H_i)V_{\max}
    + \beta\left(\frac{1}{H_{i}} - \frac{1}{H_{i+1}}\right),
\end{align*}
which is exactly the same bound as in the average cost case. Therefore, the same analysis shows that
\[
  V_{\max}(H_{i+1}-H_i)
  + \beta\left(\frac{1}{H_{i}} - \frac{1}{H_{i+1}}\right) \leq \epsilon.
\]
for all $i \geq 0$ and that we may bound 
\begin{align*}
                     K    \leq
                           \frac{\log\left(\frac{2\sqrt{V_{\max}}}{\epsilon}\right)}{\log \left( 1 +  \frac{\epsilon}{ 2\beta V_{\max}/\epsilon +\sqrt{V_{\max}}}\right)},                           
  \end{align*}
  leading to the conclusion that $K = O(\log(1/\epsilon))$.
  
\end{proof}

\begin{proof}[Proof of Theorem~\ref{thm:H_opt_discounted_cost}]
  Running the discounted SGD Algorithm (Figure~\ref{alg:SGD} with subgradient $g^\gamma(\theta)$) for $H_k$

  $H_1,\ldots,H_K$ with $4T$ steps, where $T$ is set as in Theorem~\ref{thm:main_discounted},
produces a sequence $\widehat\theta_1,\ldots,\widehat\theta_K$ such that
\[
  c(H_k,\widehat\theta_K)\leq c(H_k,\theta_K^*) + H_kV(\theta^*)+ \frac{\beta}{H_k} + \frac{\epsilon}{4}
\]
holds for all $k$ simultaneously with probability at least $1-\frac{\delta}{2}$, which is easily argued by noting that the probability of error for any single $k$ is $\delta/K$ and applying the union bound.

Lemma~\ref{lem:V_hat_error_discounted}, along with our choice of
\[
    n \geq  \frac{(S(C_3 + 2C_4))^2}{2\epsilon^2}\log\left(\frac{4K}{\delta}\right)
\]
guarantees that $\abs{V_3(\widehat\theta_k)+V_4(\widehat\theta_k) -
  \widehat V_k}\leq\frac{\epsilon}{4}$
holds with probability at least $1-\frac{\delta}{2K}$, and hence the statement holds for all $\widehat V_k$ with probability at most $1-\frac{\delta}{2}$.

We now turn to bounding the suboptimality of the objective. Recalling that $\hat k$ is the minimizer of $\ell^\top\Phi\widehat\theta_{k} + \left(H_{k}+\frac{1}{1-\gamma}\right) \widehat V_k
      + \frac{\beta}{H_{k}}$, and using $k^*$ as the minimizer of
$c(H_k,\theta_k^*) + \frac{\beta}{H_k}$, we have
    \begin{align*}
      \ell^\top\Phi\widehat\theta_{\hat k} + \left(H_{\hat k}+\frac{1}{1-\gamma}\right) \widehat V_{\hat k}
      + \frac{\beta}{H_{\hat k}}
      &=
        \min_k \ell^\top\Phi\widehat\theta_{k} + \left(H_{k}+\frac{1}{1-\gamma}\right)\widehat V_k
      + \frac{\beta}{H_{k}}\\
      &\leq
        \ell^\top\Phi\widehat\theta_{k^*} + \left(H_{k^*}+\frac{1}{1-\gamma}\right) \widehat V_{k^*}
        + \frac{\beta}{H_{k^*}}\\
      &\leq
        c(H_{k^*},\widehat\theta_{k^*}) + \frac{\beta}{H_{k^*}} + \frac{\epsilon}{4}
      &\text{(Lemma~\ref{lem:V_hat_error_discounted})}\\
      &\leq
        c(H_{k^*},\theta_{k^*}^*) + \frac{\beta}{H_{k^*}}
        + \frac{\epsilon}{2}
      &\text{(Theorem~\ref{thm:main_discounted})}\\
      &=\min_k
        c(H_k,\theta_k^*) + \frac{\beta}{H_k}
        + \frac{\epsilon}{2}\\
      &\leq  \min_{H,\theta} c(H,\theta) + \frac{\beta}{H}
        + \epsilon
      &\text{(Lemma~\ref{lem:H_grid_bound_discounted})}.
    \end{align*}
    The statement holds with probability at least $\frac{\delta}{2} + \frac{\delta}{2}$, where the first term is from estimating $\widehat V_k$ (Lemma~\ref{lem:V_hat_error_discounted}) and the second term is from bounding the SGD error (Theorem~\ref{thm:main_discounted}). Hence, the Meta-algorithm minimizes the objective to within $\epsilon$.

Next, we use Lemma~\ref{lemma:key_discounted} to bound the discrepancy between $\Phi\theta$ and $\nu_\theta$. Therefore, we need to bound $V_3(\widehat\theta_{\hat k})$ and $V_4(\widehat\theta_{\hat k})$. Since all quantities are non-negative, this implies that $\abs{\frac{\beta}{H_{\hat k}} - \frac{\beta}{H^*}}\leq\epsilon$.
Using the bounded suboptimality of $\widehat\theta_{\hat k}$ as an optimizer of $c(H_{\hat k},\theta)$, we have
\begin{align*}
  \ell^\top\Phi\widehat\theta_{\hat k}
  +\left(\frac{1}{1-\gamma} +  H_{\hat k}\right)\left(V_3(\widehat\theta_{\hat k}) + V_4(\widehat\theta_{\hat k})\right)
  &\leq
  \ell^\top\Phi\theta^*_{\hat k}
  + \left(\frac{1}{1-\gamma} +  H_{\hat k}\right)\left(V_3(\theta^*_{\hat k}) + V_4(\theta^*_{\hat k})\right)+ \frac{\epsilon}{2}\\
  & \leq
  \ell^\top\Phi\theta^*
  + \left(\frac{1}{1-\gamma} +  H^*\right)\left(V_3(\theta^*) + V_4(\theta^*)\right)+ \epsilon\\
&=
  \ell^\top\Phi\theta^*
  + \frac{1}{1-\gamma}\left(V_3(\theta^*) + V_4(\theta^*)\right)\\
&\quad +  \sqrt{V_3(\theta^*) + V_4(\theta^*)}+ \epsilon.
\end{align*}
Next, we crudely bound $\ell^\top\Phi\theta\leq \sqrt{d}CS$ and use  $\left(\frac{1}{1-\gamma} + H_{\hat k}\right)^{-1}\leq
\frac{1}{H_{\hat k}}$ to obtain
\begin{align*}
  V_3(\widehat\theta_{\hat k})+ V_4(\widehat\theta_{\hat k})
  &\leq
    \frac{1}{H_{\hat k}}
    \left(
    2\sqrt{d}CS + \sqrt{V_3(\theta^*) + V_4(\theta^*)}
    +\frac{1}{(1-\gamma)}\left(V_3(\theta^*) + V_4(\theta^*)\right)
    +\epsilon\right)\\
  &\leq
\left(\frac{1}{H^*}+\beta\epsilon\right)
    \left(
    2\sqrt{d}CS + \sqrt{V_3(\theta^*) + V_4(\theta^*)}
    +\frac{1}{(1-\gamma)}\left(V_3(\theta^*) + V_4(\theta^*)\right)
    +\epsilon\right)\\
  &\leq
2\sqrt{d}CS\sqrt{V_3(\theta^*) + V_4(\theta^*)}
    + \left(V_3(\theta^*) + V_4(\theta^*)\right)
    +\frac{\left(V_3(\theta^*) + V_4(\theta^*)\right)^{\frac{3}{2}}}{(1-\gamma)}
+ O(\epsilon).
\end{align*}

Then, applying Lemma~\ref{lemma:key_discounted}, we have
\begin{align*}
  \ell^\top\Phi\mu_{\theta_{\hat k}}
  &\leq
    \ell^\top\Phi\widehat\theta_{\hat k}
    + \frac{3}{1-\gamma}\left(2\sqrt{d}CS\sqrt{V_3(\theta^*) + V_4(\theta^*)}
    + \left(V_3(\theta^*) + V_4(\theta^*)\right)
    +\frac{\left(V_3(\theta^*) + V_4(\theta^*)\right)^{\frac{3}{2}}}{(1-\gamma)}
    \right)\\
    &\quad+ O\left(\frac{\epsilon}{1-\gamma}\right)\\
  &\leq
    \ell^\top\Phi\theta^*
    + \frac{3}{1-\gamma}\left(2\sqrt{d}CS\sqrt{V_3(\theta^*) + V_4(\theta^*)}
    + \left(V_3(\theta^*) + V_4(\theta^*)\right)
    +\frac{\left(V_3(\theta^*) + V_4(\theta^*)\right)^{\frac{3}{2}}}{(1-\gamma)}
    \right)\\
  &\quad
      + \left(\frac{1}{1-\gamma}+H^*\right)\left(V_3(\theta^*) + V_4(\theta^*)\right)
     + O\left(\frac{\epsilon}{1-\gamma}\right)\\
  &\leq
    \ell^\top\nu_{\theta^*}
    + \frac{3}{1-\gamma}\left(2\sqrt{d}CS\sqrt{V_3(\theta^*) + V_4(\theta^*)}
    + \left(V_3(\theta^*) + V_4(\theta^*)\right)
    +\frac{\left(V_3(\theta^*) + V_4(\theta^*)\right)^{\frac{3}{2}}}{(1-\gamma)}
    \right)\\
  &\quad
      + \left(\frac{1}{1-\gamma}+H^*\right)\left(V_3(\theta^*) + V_4(\theta^*)\right)
    + \epsilon + \frac{3}{1-\gamma}\left(V_3(\theta^*) + V_4(\theta^*)\right)
    + O\left(\frac{\epsilon}{1-\gamma}\right)\\
  &\leq
    \ell^\top\nu_{\theta^*}
    + \left(1+\frac{3}{1-\gamma}2\sqrt{d}CS\right)\sqrt{V_3(\theta^*) + V_4(\theta^*)}
    +\frac{3}{1-\gamma}\frac{\left(V_3(\theta^*) + V_4(\theta^*)\right)^{\frac{3}{2}}}{(1-\gamma)}\\
  &\quad
    + \frac{7}{1-\gamma}\left(V_3(\theta^*) + V_4(\theta^*)\right)+O\left(\frac{\epsilon}{1-\gamma}\right).
\end{align*}
All in all, this simplifies to
\begin{equation*}
  \ell^\top\nu_{\theta_{\hat k}} \leq \min_\theta \ell^\top\nu_\theta
  + O\left(\sqrt{V_3(\theta)+V_4(\theta)}\right)
  + O\left(\left(V_3(\theta)+V_4(\theta)\right)^{\frac{3}{2}}\right)
+ O\left(\frac{\epsilon}{1-\gamma}\right).
\end{equation*}
Using our assumption that $(V_3(\theta)+V_4(\theta))<1$, we obtain the theorem statement.

We now turn towards bounding the subgradient steps and number of samples. Since the $H_k$ are equal to the average cost case, we can still bound $K=O(\log(1/\epsilon))$. Theorem~\ref{thm:main_discounted} requires we use
\[
  T_k =
  \frac{S^2}{\epsilon^2}
  \left(H_k(C_3+C_4)+\sqrt{d}+2\sqrt{10 \log\frac{1}{\delta}}+2\sqrt{5d \log \left( 1 + \frac{S^2 T}{d} \right)}\right)^2,
\]
and so the total number of gradient descent steps can be bounded by
\begin{equation*}
  \sum_k T_k \leq K T_K
  =
  O\left(
    \frac{1}{\epsilon^4}
    \right),
  \end{equation*}
with the same number of samples as in the average cost case.
\end{proof}

\section{Related Work}
\label{sec:related_work}
One of the approximate linear programming methods, proposed by~\cite{Schweitzer-Seidmann-1985}, was to project the primal LP into a subspace. These ideas have seen lots of recent work \citep{DeFarias-VanRoy-2003,
  deFarias-VanRoy-NIPS-2003, Hauskrecht-Kveton-2003,
  Guestrin-Hauskrecht-Kveton-2004, Petrik-Zilberstein-2009,
  Desai-Farias-Moallemi-2012}. As noted by
\cite{Desai-Farias-Moallemi-2012}, the prior work on ALP either
requires access to samples from a distribution that depends on optimal
policy or assumes the ability to solve an LP with as many constraints
as states.

The first theoretical analysis of ALP methods, by \cite{DeFarias-VanRoy-2003}, analyzed the discounted primal LP \eqref{LP:exact_discounted_dual} performance
when only value functions of the form $J=\Phi w$, for some feature matrix $\Phi$, are considered. Roughly, they show that the ALP solution $w^*$ has the family of error bound indexed by a vector $u\in\Reals^\cX$
\begin{equation}
\label{eq:ALP-apprx}
\norm{J_* - \Psi w_*} \le
\frac{2 c^\top u}{1-\beta_{u}} \min_{w} \norm{J_* - \Psi w}_{\infty, 1/u}
\end{equation}
where $c$ is a ``state-relevance'' vector and $\beta_u = \gamma \max_{x,a} \sum_{x'} P_{(x,a), x'} u(x')/u(x)$ is a ``goodness-of-fit'' parameter that measures how well $u$ represents a stationary distribution. Unfortunately, $c$ and $u$ are typically hard to choose (for example, a good choice of $c$ would be the stationary distribution under $w^*$, which we do not know); but more importantly, the bound can be vacuous if $\Psi$ does not model the optimal value function well and $\norm{J_* - \Psi w}$ is always large. In particular, the problem we are considering in Definition~\ref{defn:E.ELALP} requires an additive bound with respect to the optimal parameter.

This result has some limitations. We need to specify $c$, but a good choice is usually not known a priori. The authors show that, if the ALP is solved iteratively using the $c=\mu_{\pi_{\Psi w_*}, \nu}$ from the last iteration, then for an arbitrary probability distribution $\nu\in \Delta_{[\cX]}$ and accompanying $\mu_{\pi, \nu} = (1-\gamma) \nu^\top (I - \gamma P^\pi)^{-1}$, we must have
\begin{equation*}
\norm{J_{\pi_J} - J_*}_{1,\nu} \le \frac{1}{1-\gamma} \norm{J - J_*}_{1, \mu_{\pi_J, \nu}} ,
\end{equation*}
where $J_*$ is the discounted cost of the optimal policy. This suggests that we should choose $c=\mu_{\pi_{\Psi w_*}, \nu}$, which is impossible as $w_*$ is not known a priori.

A second limitation is that the ALP remains computationally expensive if the number of constraints is large and was addressed in \cite{DeFarias-VanRoy-2004} by reducing the number of constraints by sampling them. The idea is to sample a relatively small number of constraints and solve the resulting LP.
Let $\mathcal N\subset \Reals^d$ be a known set that contains $w_*$ (solution of ALP). Let $\mu_{\pi, c}^V (x) = \mu_{\pi, c}(x) V(x)/(\mu_{\pi, c}^\top V)$ and define the distribution $\rho_{\pi, c}^V(x,a) = \mu_{\pi, c}^V (x) / \cA$. Let $\delta\in (0,1)$ and $\epsilon\in (0,1)$. Let $\overline\beta_u = \gamma \max_{x} \sum_{x'} P_{(x,\pi_*(x)), x'} u(x')/u(x)$ and
\[
D = \frac{(1+\overline\beta_V) \mu_{\pi_*, c}^\top V}{2 c^\top J_*} \sup_{w\in \mathcal N} \norm{J_* - \Psi w}_{\infty, 1/V}\,, \qquad m \ge \frac{16 \cA D}{(1-\gamma)\epsilon} \left( d \log \frac{48 \cA D}{(1-\gamma) \epsilon} + \log\frac{2}{\delta} \right) \; .
\]
Let $\mathcal S$ be a set of $m$ random state-action pairs sampled under $\rho_{\pi_*, c}^V$. Let $\widehat w$ be a solution of the following sampled LP:
\begin{align*}
&\max_{w\in\Reals^d}\  c^\top \Psi w\,,  \\
&\mbox{s.t.}\quad  w\in \mathcal N,\, \forall (x,a)\in \mathcal S,\, \ell(x,a) + \gamma P_{(x,a),:} \Psi w \ge (\Psi w)(x) \; .
\end{align*}
\citet{DeFarias-VanRoy-2004} prove that with probability at least $1-\delta$, we have
\[
\norm{J_* - \Psi \widehat w}_{1,c} \le \norm{J_* - \Psi  w_*}_{1,c} + \epsilon \norm{J_*}_{1,c} \; .
\]
Unfortunately, $\mu_{\pi_*, c}$ (which was used in the definition of $D$) depends on the optimal policy, which is obviously unknown, which makes this method difficult to implement.

In the primal form \eqref{LP:exact_average}, an extra constraint $h = \Psi w$ is added to obtain
\begin{align}
\label{eq:primal-apprx}
&\max_{\lambda, w} \lambda\,, \\
\notag
&\mbox{s.t.}\quad B(\lambda e + \Psi w) \ge \ell + P \Psi w \; .
\end{align}

Let $\lambda_*$ be the average loss of the optimal policy and $(\widetilde\lambda, \widetilde w)$ be the solution of this LP. It turns out that the greedy policy with respect to $\widetilde w$ can be arbitrarily bad even if $\abs{\lambda_* - \widetilde \lambda}$ was small~\citep{deFarias-VanRoy-NIPS-2003}. \citet{deFarias-VanRoy-NIPS-2003} propose a two stage procedure, where the above LP is the first stage and the second stage is
\begin{align}
\label{eq:second-stage}
\notag
&\max_{w} c^\top \Psi w\,, \\
&\mbox{s.t.}\quad B(\widetilde\lambda e + \Psi w) \le \ell + P \Psi w \, ,
\end{align}
where $c$ is a user specified weight vector. Let $\widehat w$ be the solution of the second stage. Let $\lambda_w$ and $\mu_w$ be the average loss and the stationary distribution of the greedy policy with respect to $\Psi w$. \citet{deFarias-VanRoy-NIPS-2003} prove that
\[
\lambda_w - \lambda_* \le \norm{h_* - \Psi w}_{1, \mu_w} \; .
\]
Further, it is shown that $\widehat w$ minimizes $\norm{h_{\widetilde\lambda} - \Psi w}_{1, c}$ and that
\[
\norm{h_* - \Psi \widehat w}_{1, c} \le \norm{h_{\widetilde\lambda} - \Psi \widehat w}_{1, c} + (\lambda_* - \widetilde \lambda) c^\top (I - P^{\pi_*})^{-1} e\,,
\]
which implies that $\norm{h_* - \Psi \widehat w}_{1, c}$ is small. To get that $\lambda_{\widehat w} - \lambda_*$ is small, we need to use $c = \mu_{\widehat w}$. Value of  $\mu_{\widehat w}$ is obtained only after solving the optimization problem \eqref{eq:second-stage}. To fix this problem, \citet{deFarias-VanRoy-NIPS-2003} propose to solve \eqref{eq:second-stage} iteratively, using $c = \mu_{\widehat w}$ from the solution of the last round.

The above approach has two problems. First, it is still not clear if the average loss of the resultant policy is close to $\lambda_*$ (or the best policy in the policy class). Second, iteratively solving \eqref{eq:second-stage} is computationally expensive. Similar results are also obtained by \citet{Desai-Farias-Moallemi-2012} who also show that if we were able to sample from the stationary distribution of the optimal policy, then LP~\eqref{eq:primal-apprx} can be solved efficiently.

\citet{Desai-Farias-Moallemi-2012} study a smoothed version of ALP, in which slack variables are introduced that allow for some violation of the constraints. Let $D'$ be a violation budget. The smoothed ALP (SALP) has the form of
\begin{align*}
&\max_{w,s} c^\top \Psi w\,, &\max_{w,s}&\, c^\top \Psi w - \frac{2 \mu_{\pi_*, c}^\top s}{1-\gamma}\,,  \\
&\mbox{s.t.}\quad \Psi w \le L \Psi w + s,\, \mu_{\pi_*, c}^\top s\le D',\, s\ge {\mathbf 0},\, &\mbox{s.t.}\quad & \Psi w \le L \Psi w + s,\, s\ge {\mathbf 0} \; .
\end{align*}
The ALP on RHS is equivalent to LHS with a specific choice of $D'$. Let $\overline{U}=\{u\in\Reals^\cX\ : \ u\ge \mathbf{1} \}$ be a set of weight vectors. \citet{Desai-Farias-Moallemi-2012} prove that if $w_*$ is a solution to above problem, then
\[
\norm{J_* - \Psi w_*}_{1,c} \le \inf_{w, u\in \overline{U}} \norm{J_* - \Psi w}_{\infty, 1/u} \left( c^\top u + \frac{2(\mu_{\pi_*, c}^\top u) (1+ \beta_u)}{1-\gamma} \right) \; .
\]
The above bound improves \eqref{eq:ALP-apprx} as $\overline{U}$ is larger than $U$ and RHS in the above bound is smaller than RHS of \eqref{eq:ALP-apprx}. Further, they prove that if $\eta$ is a distribution and we choose $c=(1-\gamma) \eta^\top (I-\gamma P^{\pi_{\Psi w_*}})$, then
\[
\norm{J_{\mu_{\Psi w_*}}-J_*}_{1,\eta} \le  \frac{1}{1-\gamma} \left( \inf_{w, u\in \overline{U}} \norm{J_* - \Psi w}_{\infty, 1/u} \left( c^\top u + \frac{2(\mu_{\pi_*, \nu}^\top u) (1+\beta_u)}{1-\gamma} \right)  \right) \; .
\]
Similar methods are also proposed by \citet{Petrik-Zilberstein-2009}.
One problem with this result is that $c$ is defined in terms of $w_*$, which itself depends on $c$. 
Also, the smoothed ALP formulation uses $\pi_*$ which is not known. \citet{Desai-Farias-Moallemi-2012} also propose a computationally efficient algorithm. Let $\mathcal S$ be a set of $S$ random states drawn under distribution $\mu_{\pi_*, c}$. Let $\mathcal N'\subset \Reals^d$ be a known set that contains the solution of SALP. The algorithm solves the following LP:
\begin{align*}
&\max_{w,s}\, c^\top \Psi w - \frac{2}{(1-\gamma) S}  \sum_{x\in \mathcal S} s(x)\,,  \\
&\mbox{s.t.}\quad \forall x\in \mathcal S,\, (\Psi w)(x) \le (L \Psi w)(x) + s(x),\, s\ge \mathbf 0,\, w\in \mathcal N' \; .
\end{align*}
Let $\widehat w$ be the solution of this problem.
\citet{Desai-Farias-Moallemi-2012} prove high probability bounds on the approximation error $\norm{J_* - \Psi \widehat w}_{1,c}$. However, it is no longer clear if a performance bound on $\norm{J_* - J_{\pi_{\Psi \widehat w}}}_{1,c}$ can be obtained from this approximation.


Next, we turn our attention to average cost ALP. Let $\nu$ be a distribution over states, $u:[\cX]\rightarrow [1,\infty)$, $\eta>0$, $\gamma\in [0,1]$, $P_\gamma^\pi = \gamma P^\pi + (1-\gamma) \mathbf{1} \nu^\top$, and $L_\gamma h = \min_{\pi} (\ell_\pi + P_\gamma^\pi h)$.
\citet{DeFarias-VanRoy-2006} propose the following optimization problem:
\begin{align}
\label{eq:primal-approx1}
&\min_{w, s_1, s_2} s_1 + \eta s_2\,, \\
\notag
&\mbox{s.t.}\quad L_\gamma \Psi w - \Psi w + s_1 \mathbf{1} + s_2 u \ge \mathbf 0,\, s_2 \ge 0 \; .
\end{align}
Let $(w_*, s_{1,*}, s_{2,*})$ be the solution of this problem.
Define the mixing time of policy $\pi$ by
\[
\tau_\pi = \inf\left\{ \tau \ : \ \abs{\frac{1}{t} \sum_{t'=0}^{t-1} \nu^\top (P^\pi)^{t'} \ell_\pi - \lambda_\pi   } \le \frac{\tau}{t},\, \forall t \right\} \; .
\]
Let $\tau_* = \liminf_{\delta\rightarrow 0} \{ \tau_\pi : \lambda_\pi \le \lambda_* + \delta \}$. Let $\pi_{\gamma}^*$ be the optimal policy when discount factor is $\gamma$. Let $\pi_{\gamma, w}$ be the greedy policy with respect to $\Psi w$ when discount factor is $\gamma$, $\mu_{\gamma, \pi}^\top = 	(1-\gamma) \sum_{t=0}^\infty \gamma^t \nu^\top (P^\pi)^t$ and $\mu_{\gamma, w} = \mu_{\gamma, \pi_{\gamma, w}}$.
 \citet{DeFarias-VanRoy-2006} prove that if $\eta \ge (2-\gamma) \mu_{\gamma, \pi_{\gamma}^*}^\top u$,
\[
\lambda_{w_*} - \lambda_* \le \frac{(1+\beta) \eta \max(D'', 1)}{1-\gamma} \min_{w} \norm{h_\gamma^* - \Psi w}_{\infty, 1/u} + (1-\gamma) (\tau_* + \tau_{\pi_{w_*}}) \,,
\]
where $\beta = \max_\pi \norm{I - \gamma P^\pi}_{\infty, 1/u}$, $D'' = \mu_{\gamma, w_*}^\top V /  (\nu^\top V)$ and $V=L_\gamma \Psi w_* - \Psi w_* + s_{1,*} \mathbf{1} +  s_{2,*} u$. Similar results are obtained more recently by \citet{Veatch-2013}.

An appropriate choice for vector $\nu$ is $\nu=\mu_{\gamma, w_*}$. Unfortunately, $w_*$ depends on $\nu$. We should also note that solving \eqref{eq:primal-approx1} can be computationally expensive. \citet{DeFarias-VanRoy-2006} propose constraint sampling techniques similar to \citep{DeFarias-VanRoy-2004}, but no performance bounds are provided.

\end{document}

%% file: def.tex
\let\inf\undefined
\DeclareMathOperator*{\inf}{\vphantom{\sup}inf}
\DeclareMathOperator*{\ex}{\mathbb E}

\DeclareMathOperator*{\sgn}{sgn}

\DeclareBoldMathCommand{\b}{b}
\DeclareBoldMathCommand{\c}{c}
\DeclareBoldMathCommand{\d}{d}
\DeclareBoldMathCommand{\e}{e}
\DeclareBoldMathCommand{\f}{f}
\DeclareBoldMathCommand{\g}{g}
\DeclareBoldMathCommand{\n}{n}
\DeclareBoldMathCommand{\m}{m}
\DeclareBoldMathCommand{\p}{p}
\DeclareBoldMathCommand{\q}{q}
\DeclareBoldMathCommand{\r}{r}
\DeclareBoldMathCommand{\s}{s}

\DeclareBoldMathCommand{\v}{v}
\DeclareBoldMathCommand{\w}{w}
\DeclareBoldMathCommand{\z}{z}

\DeclareBoldMathCommand{\A}{A}
\DeclareBoldMathCommand{\B}{B}
\DeclareBoldMathCommand{\C}{C}
\DeclareBoldMathCommand{\D}{D}
\DeclareBoldMathCommand{\F}{F}
\DeclareBoldMathCommand{\G}{G}
\DeclareBoldMathCommand{\H}{H}
\DeclareBoldMathCommand{\I}{I}
\DeclareBoldMathCommand{\J}{J}
\DeclareBoldMathCommand{\K}{K}
\DeclareBoldMathCommand{\L}{L}
\DeclareBoldMathCommand{\M}{M}
\DeclareBoldMathCommand{\P}{P}
\DeclareBoldMathCommand{\Q}{Q}
\DeclareBoldMathCommand{\R}{R}
\DeclareBoldMathCommand{\S}{S}

\DeclareBoldMathCommand{\V}{V}
\DeclareBoldMathCommand{\U}{U}
\DeclareBoldMathCommand{\W}{W}
\DeclareBoldMathCommand{\X}{X}
\DeclareBoldMathCommand{\Y}{Y}
\DeclareBoldMathCommand{\Z}{Z}

\DeclareBoldMathCommand{\ssigma}{\sigma}
\DeclareBoldMathCommand{\SSigma}{\Sigma}
\DeclareBoldMathCommand{\OOmega}{\Omega}

\DeclareBoldMathCommand{\mmu}{\mu}
\DeclareBoldMathCommand{\ones}{1}
\DeclareBoldMathCommand{\zeros}{0}

\let\top\intercal

\newcommand{\norm}[1]{\left\Vert#1\right\Vert}

\newcommand{\abs}[1]{\left\vert#1\right\vert}

\newcommand{\nat}{\mathbb N}                   

\newcommand{\one}[1]{{\mathbb I}{\{#1\}}}           

\newcommand{\cZ}{{\cal Z}}
\newcommand{\cX}{{\cal X}}

\newcommand{\cA}{{\cal A}}

\newcommand{\Reals}{\mathbb R}

\DeclareMathOperator*{\argmin}{arg\,min}

\newcommand{\TODO}[1]{
\ifmmode
\text{\textcolor{red}{TODO: #1}}
\else
\textcolor{red}{TODO: #1}
\fi
}